\documentclass{amsart}
\usepackage{amsmath,amssymb,tikz}
\usetikzlibrary{arrows,shapes,positioning,decorations.markings}
\usepackage[margin=1.5cm]{geometry}
\newtheorem{theorem}{Theorem}[section]
\newtheorem{lemma}[theorem]{Lemma}
\newtheorem{corollary}[theorem]{Corollary}
\newtheorem{conjecture}[theorem]{Conjecture}
\newtheorem{proposition}[theorem]{Proposition}

\newtheorem{remark}[theorem]{Remark}

\newtheorem{definition}[theorem]{Definition}
\newcommand{\CD}{\mathrm{CD}}
\newcommand{\UDRR}{\mathrm{DRR}}
\newcommand{\C}{\mathrm{C}}
\newcommand{\vGa}{{\Gamma}}
\renewcommand{\wr}{\mathop{\mathrm{wr}}}

\newcommand{\Alt}{\mathop{\mathrm{Alt}}}
\newcommand{\Sym}{\mathop{\mathrm{Sym}}}
\newcommand{\Cay}{\mathop{\Gamma}}
\newcommand{\Aut}{\mathop{\mathrm{Aut}}}
\newcommand{\Out}{\mathop{\mathrm{Out}}}

\newcommand{\PSL}{\mathop{\mathrm{PSL}}}

\def\cent#1#2{{\bf C}_{#1}(#2)}
\def\norm#1#2{{\bf N}_{#1}(#2)}

\makeatletter
\newcommand{\myitem}[1]{%
\item[#1]\protected@edef\@currentlabel{#1}%
}
\makeatother

\begin{document}

\title[Enumeration of Cayley digraphs]{Asymptotic enumeration of Cayley digraphs}

\author[J. Morris]{Joy Morris}
\address{Department of Mathematics and Computer Science,
University of Lethbridge,\newline Lethbridge, AB. T1K 3M4. Canada.}
\email{joy.morris@uleth.ca}
\author[P. Spiga]{Pablo Spiga}
\address{Pablo Spiga,
Dipartimento di Matematica e Applicazioni, University of Milano-Bicocca,\newline
Via Cozzi 55, 20125 Milano, Italy}\email{pablo.spiga@unimib.it}

\thanks{Supported in part by the National Science and Engineering Research Council of Canada. Address correspondence to P. Spiga,
E-mail: pablo.spiga@unimib.it.}

\begin{abstract}
In this paper we show that almost all Cayley digraphs have automorphism group as small as possible; that is, they are digraphical regular representations (DRRs). More precisely, we show that as $r$ tends to infinity, for every finite group $R$ of order $r$, out of all possible Cayley digraphs on $R$ the proportion whose automorphism group is as small as possible tends to $1$. This proves a natural conjecture first proposed in $1982$ by Babai and Godsil.
\end{abstract}

\keywords{regular representation, Cayley graph, automorphism group, asymptotic enumeration, graphical regular representation, DRR, GRR}
\maketitle

\section{Introduction}\label{intro}

\subsection{Background and significance}

All digraphs and groups considered in this paper are finite. By a {\em digraph} $\vGa$, we mean an ordered pair $(V,A)$ where the \emph{vertex-set} $V$ is a finite non-empty set and the {\em arc-set} $A\subseteq V\times V$ is a binary relation on $V$. The elements of $V$ and $A$ are called \emph{vertices} and \emph{arcs} of $\vGa$, respectively. An \emph{automorphism} of $\vGa$ is a permutation $\sigma$ of $V$ that preserves the relation $A$, that is, $(x^\sigma,y^\sigma)\in A$ for every $(x,y)\in A$.

Let $R$ be a group and let $S$ be a subset of $R$. The \emph{Cayley digraph} on $R$ with connection set $S$, denoted $\Cay(R,S)$, is the digraph with vertex-set $R$ and with $(g,h)$ being an arc if and only if $hg^{-1}\in S$. Note that we do not require our Cayley digraphs to be connected and that they may have loops. It is an easy observation that $R$ acts regularly as a group of automorphisms of $\Cay(R,S)$ by right multiplication and hence $R\le \Aut(\Cay(R,S))$.

Although we have just seen that the definition of a Cayley digraph forces the automorphism group of such a digraph to contain a group acting regularly, there is nothing in the definition that tells us whether or not such a digraph has any other automorphisms. When considering questions of structure and isomorphism, determining the full automorphism group of a digraph is a very important question. In the case of a Cayley digraph on a group $R$, a major first step in finding the answer to this question is to determine whether $R$ is in fact the full automorphism group of this digraph. When it is, $\Cay(R,S)$ is called a \emph{DRR} (for digraphical regular representation). 

In addition to the value of determining the full automorphism group of Cayley digraphs, DRRs are of considerable interest in that such a digraph provides a visual representation of the group on which it was defined. In this way, for example, the cyclic group of order $n$ can be introduced as the group of symmetries of a directed $n$-gon.

Babai and Godsil made the following conjecture.

\begin{conjecture}[\cite{Go2}, Conjecture 3.13; \cite{BaGo}]\label{digraphmainconjecture}
Let $R$ be a group of order $r$. The proportion of subsets $S$ of $R$ such that $\Cay(R,S)$ is a $\mathrm{DRR}$ goes to $1$ as $r\to\infty$.
\end{conjecture}

In other words, almost all Cayley digraphs are DRRs, in the sense that
$$\lim_{|R|\to\infty}\frac{|\{S\subseteq R\mid \Aut(\Cay(R,S))=R
\}|}{2^{|R|}}=1.$$
  Godsil showed that Conjecture~\ref{digraphmainconjecture} holds if $G$ is a $p$-group with no homomorphism onto the wreath product $\C_p \wr \C_p$~\cite{Go2}, and Babai and Godsil extended this to verify the conjecture in the case that $G$ is nilpotent of odd order~\cite[Theorem 2.2]{BaGo}.  This paper gives a proof of the full Conjecture~\ref{digraphmainconjecture} of Babai and Godsil. 

\begin{theorem}\label{th:main1}
Let $R$ be a group of order $r$.  The proportion of subsets $S$ of $R$ such that $\Cay(R,S)$ is a $\mathrm{DRR}$ goes to $1$ as $r\to\infty$.
\end{theorem}

Actually, we prove a quantified version of this result, which might have some independent interest and might be useful in some applications.

\begin{theorem}\label{th:main2}
Let $R$ be a group of order $r$, where $r$ is sufficiently large.  The number of subsets $S$ of $R$ such that $\Cay(R,S)$ is not a $\mathrm{DRR}$ is at most $2^{r-br^{0.499}/(4(\log_2(r))^3)+2}$, where $b$ is an  absolute constant that does not depend on $R$.
\end{theorem}

This quantified version makes it clear that our results also resolve the directed version of Xu's conjecture about normal Cayley graphs. (A normal Cayley (di)graph $\Cay(R,S)$ is a Cayley (di)graph having the property that $R \trianglelefteq \Aut(\Cay(R,S))$, so every DRR is a normal Cayley digraph.)

\begin{theorem}[Conjecture 1, \cite{Xu1998}]\label{th:normal}
The minimum over all groups $R$ of order $r$ of the proportion of subsets $S$ of $R$ such that $\Cay(R,S)$ is a normal Cayley digraph tends to $1$ as $r \to\infty$.
\end{theorem}

It is well-known that almost all graphs (and almost all digraphs) are asymmetric. The graph version of this result (which is more difficult than the digraph version) follows from P\'{o}lya enumeration (or Burnside's counting lemma), as is mentioned in \cite{FordUhlenbeck} without proof. A full proof along these lines is given in \cite{Harary}; such a proof can also be found in~\cite[Section~2.3]{GodsilRoyle}. An alternative proof derives from work by Erd\"{o}s and R\'enyi \cite{ErdosRenyi}, who obtained formulas for the number of graphs of order $n$ admitting a given permutation of degree $n$ as an automorphism  (the formulas depend on the number of fixed points of the permutation); combining these formulas over all possible permutations implies that almost all graphs of order $n$ are asymmetric.

This makes the Babai-Godsil conjecture (and hence our theorem) very natural: nature seems to be typically meagre and rarely gives more than is truly necessary. Thus, even though the digraph $\Cay(R,S)$ is constructed in such a way as to force it to contain $R$ in its automorphism group, only exceptionally does $\Cay(R,S)$ admit any extra automorphisms (beyond those that have been forced by the construction). 

Let $\CD(R)$ denote the set of  Cayley digraphs over  $R$ up to isomorphism and let $\UDRR(R)$ denote the set of DRRs over $R$ up to isomorphism. We also provide a proof of the following unlabelled version of Theorem~\ref{th:main1}. Formally, $\CD(R)$ is a set of representatives for the equivalent relation on $\{\Cay(R,S)\mid S\subset R\}$ given by being isomorphic, and $\UDRR(R)$ consists of the elements of $\CD(R)$ which are DRRs.

\begin{theorem}\label{th:unlabelledmain1}
Let $R$ be a  group of order $r$. Then $|\UDRR(R)|/|\CD(R)|$ tends to $1$ as $r\to\infty$.
\end{theorem}

\subsection{Notation, and outline of the proof of Theorem~\ref{th:main1}}

Given a group $G$ and $g\in G$, we write $o(g)$ for the order of $g$.

Throughout this paper, we denote by $\Sym(r)$ and by $\Sym(\Omega)$ the symmetric group of degree $r$ and the symmetric group on the set $\Omega$. We will use the first notation when the underlying point set is irrelevant for our investigation, and we will use the second notation otherwise. Moreover, we let $R$ denote a group of order $r$. We identify $R$ with its image in $\Sym(r)$ via the right regular representation. In particular, $R\leq \Sym(r)$.

In Section~\ref{BG} we present important reductions from the work of Babai and Godsil~\cite[Section~$4$]{BaGo} that are also needed in our proof. Since our needs are slightly different from theirs, our statements also differ, so we present full proofs of our statements.

In Section~\ref{sec2}, Lemma~\ref{lemma1} shows that the problem of enumerating the subsets $S$ of $R$ with $\Aut(\Cay(R,S))=R$ reduces to obtaining an upper bound for the number of subgroups $G$ of $\Sym(r)$ with $R<G$ and with $R$ maximal in $G$. Under this reduction, $G$ is filling the role of a possible subgroup of the full automorphism group for some Cayley digraph on $R$. We use this correspondence in several of our results, to pass from counting such groups $G$, to counting subsets $S$ of $R$ such that $\Aut(\Cay(R,S))$ contains some subgroup $G$ with $R$ maximal in $G$ and other specified attributes.
 
 In the first significant result of this type, we use a rather deep result of Lubotzky~\cite{Lub} to show that the number of subsets $S$ such that $\Aut(\Cay(R,S))$ admits such a $G$ where $|G|$ is ``small'' (that is, $|G:R|$ is at most $2^{r^{0.499}}$, say) is $2^{f(r)}$ for some function $f(r)$ such that $f(r)-r \to -\infty$ as $r \to \infty$, so that $2^{f(r)}$ is a vanishingly small proportion of all possible connection sets. This leaves us the problem of considering groups $G$ with $|G:R|$ rather large.

Given $G$ with $R< G\leq \Sym(r)$, we denote by $G_R$ the core of $R$ in $G$, that is, $G_R=\bigcap_{g\in G}R^g$.  Now, since $R$ is maximal in $G$, we can view $G/G_R$ as a primitive permutation group with stabiliser $R/G_R$. We are able to show that the number of subsets $S$ of a given group $R$ such that $\Aut(\Cay(R,S))$ admits such a $G$ where $|G|$ is ``large" ($|G:R|$ is greater than $2^{r^{0.499}}$, say) and the order of the core $G_R$ is ``large" (greater than $4\log_2(|R|)$) is also a vanishingly small proportion of all possible connection sets. Again (to be more precise) we show that this number is $2^{f(r)}$ for some function $f(r)$ such that $f(r)-r \to -\infty$ as $r \to \infty$.

This leaves us needing to deal with counting the connection sets $S$ such that every such $|G|$ is ``large" but in all cases, the core $G_R$ is ``small". This is the heart of this paper, and comprises a majority of the content. In order to deal with this situation, we divide these connection sets up according to two main cases. We establish these cases and provide some groundwork in Section~\ref{sec44}. In Section~\ref{sec:stopit} we deal with the case where there is a $G$ with $G_R=1$ so that $G$ is acting faithfully and primitively on the set of right cosets of $R$ in $G$. Our analysis involves case-by-case study of the various types of primitive permutation groups according to the O'Nan-Scott classification. Then in Section~\ref{sec:imprimitive} we deal with the case where there is a $G$ with $G_R>1$. In such a case we can define a particular type of quotient graph that allows us to establish a bound based on the previously discussed cases. The fact that the previous situations produced vanishingly small proportions of non-DRRs implies that the new case also produces a vanishingly small proportion of non-DRRs.

\subsection{Some key ideas, and outline of the end of the paper}

In what follows we use repeatedly the following facts.
\begin{remark}\label{rem : 1}{\rm 
\begin{enumerate}
\item Let $X$ be a finite group. Since a chain of subgroups of $X$ has length at most $\log_2(|X|)$, $X$ has a generating set of cardinality at most $\lfloor \log_2(|X|)\rfloor\le \log_2(|X|)$.

\item Any automorphism of $X$ is uniquely determined by its action on the elements of a generating set for $X$. Therefore $|\Aut(X)|\le |X|^{\lfloor\log_2(|X|)\rfloor}\le 2^{(\log_2(|X|))^2}$.

\item Let $g$ be a permutation of the finite set $\Omega$ and  take $\Delta:=\{\omega\in \Omega\mid \omega^g=\omega\}$. Then $g$ fixes each point of $\Delta$ and the cycles of $g$ on $\Omega\setminus \Delta$ have length at least $2$. Therefore $g$ fixes setwise at most $2^{|\Delta|+\frac{|\Omega\setminus\Delta|}{2}}$ subsets of $\Omega$. In particular, if 
$|\Delta|\le |\Omega|/2$, then 
$g$ fixes setwise at most 
$2^{\frac{3}{4}|\Omega|}$ subsets of 
$\Omega$.
\end{enumerate}}
\end{remark}

After the proof of Theorem~\ref{th:main1} is completed at the end of Section~\ref{imprimitive}, we turn to the case of unlabelled digraphs and explain in Section~\ref{sec:unlabelled} how Theorem~\ref{th:main1} implies the corresponding result Theorem~\ref{th:unlabelledmain1} for unlabelled digraphs. 

We conclude the paper with additional remarks about our proof strategy and other possible generalisations; these form Section~\ref{sec:comments}.

\section{Babai--Godsil estimates: first reduction}\label{BG}

The argument in this section is completely inspired by, and in part taken from~\cite[Section~$4$]{BaGo}. For most of the arguments in this section we could simply refer to~\cite[Section~$4$]{BaGo}, however the hypotheses there are slightly stronger than our current needs. Therefore, rather than pointing out which parts in~\cite[Section~4]{BaGo} need to be refined (and how to refine them), for the sake of completeness we make this section self-contained and repeat some parts of the results of~\cite[Section~4]{BaGo}.

Henceforth, let $R$ be a regular permutation group acting on $\{1,\ldots,r\}$. Let $N$ denote a non-identity proper normal subgroup of $R$. Let $n:=|N|$ and $b:=|R:N|=r/n$. We let $\gamma_1,\ldots,\gamma_b$ be coset representatives of $N$ in $R$. Moreover, we choose $\gamma_1:=1$ to be the identity in $R$. Observe that  $R/N$ defines a group structure on $\{1,\ldots,b\}$ by setting $ij=k$ for every $i,j,k\in \{1,\ldots,b\}$  with $\gamma_i N\gamma_j N=\gamma_k N$. 

Write $v_0:=1$ where $v_0$ has to be understood as a point in the set $\{1,\ldots,r\}$.
For each $i\in \{1,\ldots,b\}$, set $\mathcal{O}_i:={v_0}^{\gamma_iN}$. Observe that the $\mathcal{O}_i$s are the orbits of $N$ on $\{1,\ldots,r\}$, the group $N$ acts regularly on $\mathcal{O}_i$ and $|\mathcal{O}_i|=|N|=n$.

For a subset $S$ of $R$, we let $\Gamma:=\Cay(R,S)$ be the Cayley digraph of $R$ with connection set $S$, and we denote by $F_S$ the largest subgroup of $\Aut(\Gamma)$ under which each orbit of $N$ is invariant. In symbols we have
$$F_S:=\{g\in\Aut(\Gamma)\mid \mathcal{O}_i^g=\mathcal{O}_i,\textrm{ for each }i\in \{1,\ldots,b\}\}.$$
(The subscript $S$ in $F_S$ will make some of the later notation cumbersome to use, but it constantly emphasises that the definition of ``$F$'' depends on $S$.)

For a subgroup $H$ of $\Sym(r)$  and an $H$-invariant subset $X$ of $\{1,\ldots,r\}$, we write $H\vert_{ X}$ for the restriction of $H$ to $X$, that is, the image of the natural homomorphism $H\to \Sym(X)$ restricting a permutation of $H$ to $X$. For $1\leq i\leq b$, set $S_i:=S\cap \mathcal{O}_i$ and let $F_S^i:=(F_S)_{v_0}\mid_{\mathcal{O}_i}$ denote the restriction   to $\mathcal{O}_i$ of the stabiliser $(F_S)_{v_0}$ in $F_S$ of the point $v_0\in\mathcal{O}_1$.

\begin{lemma}{{{\rm (See~\cite[Lemma $4.1$]{BaGo}.)}}}\label{lemma41}
If none of the $S_i$, $i\in \{2,\ldots,b\}$, is invariant under any non-identity element of the group $F_S^i$, then $F_S=N$. 
\end{lemma}
\begin{proof}
Clearly, $N\leq F_S$ and, from the Frattini argument, $F_S=(F_S)_{v_0}N$. Fix $i\in \{2,\ldots,b\}$. Let $f\in (F_S)_{v_0}$. Since $f\in \Aut(\Gamma)$, we have $S^f=S$ and, since $f$ fixes every $N$-orbit setwise, we have $S_i^f=S_i$. Therefore, by hypothesis, the permutation $f$ restricted to $\mathcal{O}_i$ is the identity. Since this holds for each $i\in \{2,\ldots,b\}$, $f$ fixes ${\{1,\ldots,r\}\setminus \mathcal{O}_1}$ pointwise. Since this holds for every element $f\in (F_S)_{v_0}$, $(F_S)_{v_0}$ fixes $\{1,\ldots,r\}\setminus \mathcal{O}_1$ pointwise. In particular, $(F_S)_{v_0}\leq (F_S)_{v_0^{\gamma_2}}$ and, as $(F_S)_{v_0}$ and $(F_S)_{v_0^{\gamma_2}}$ have the same order, $(F_S)_{v_0}=(F_S)_{v_0^{\gamma_2}}$. 

Finally, as $(F_S)_{v_0}$ fixes  $\{1,\ldots,r\}\setminus\mathcal{O}_1$ pointwise,  $((F_S)_{v_0})^{\gamma_2}=(F_S)_{v_0^{\gamma_2}}$ fixes $(\{1,\ldots,r\}\setminus\mathcal{O}_1)^{\gamma_2}=\{1,\ldots,r\}\setminus\mathcal{O}_2$ pointwise. Thus $(F_S)_{v_0}$ fixes $(\{1,\ldots,r\}\setminus\mathcal{O}_1)\cup (\{1,\ldots,r\}\setminus\mathcal{O}_2)=\{1,\ldots,r\}$ pointwise, so $(F_S)_{v_0}=1$. Therefore $F_S=(F_S)_{v_0}N=N$.
\end{proof}

In the following lemma we slightly generalise the previously-known result, and we make the estimates in the statement of~\cite[Lemma $4.2$]{BaGo} more explicit.

\begin{lemma}{{{\rm (See~\cite[Lemma $4.2$]{BaGo}.)}}}\label{lemma42} For each $i\in \{2,\ldots,b\}$, 
$$|\{S\subseteq R\mid \textrm{there exists }f\in F_S\cap \norm{\Aut(\Gamma)}{N} \textrm{ with }v_0^f=v_0,\textrm{ and }f\vert_{\mathcal{O}_i}\neq 1\}|\leq 2^{r-\frac{n}{4}+(\log_2(n))^2+\log_2(n)}.$$
\end{lemma}
\begin{proof}
 Fix $i\in \{2,\ldots,r\}$ and denote by $\Phi_i$ the set $$\Phi_i:=\{S\subseteq R\mid \textrm{there exists }f\in F_S\cap \norm {\Aut(\Gamma)}{N}\textrm{ with }v_0^f=v_0 \textrm{ and }f\vert_{\mathcal{O}_i}\neq 1\}.$$ We follow the proof in~\cite[Lemma~$4.2$]{BaGo} (without assuming that $R$ is nilpotent of odd order). 

Let $L^i$ denote the normaliser in $\Sym(\mathcal{O}_i)$ of $N^i:=N\vert_{\mathcal{O}_i}$. The group $N^i$ acts regularly on $\mathcal{O}_i$ and is contained in $L^i$, therefore, by the Frattini argument, $L^i=N^i(L^i)_v$, where $v\in \mathcal{O}_i$. The action of $N^i$ on $\mathcal{O}_i$ is permutation isomorphic to the action of $N^i$ on itself by right multiplication and, as $N^i\unlhd L^i$, the action of $(L^i)_v$ on $\mathcal{O}_i$ is permutation isomorphic to the action of $(L^i)_v$ on $N^i$ by conjugation. Therefore $L^i$ is isomorphic to the holomorph $N\rtimes \Aut(N)$ of $N$ and $$|L^i|= |N||\Aut(N)|< n\cdot n^{\log_2(n)}.$$

We claim that  if $\ell\in L^i$ fixes $v\in \mathcal{O}_i$, then the set $\{m\in N\mid v^{m\ell}=v^m\}$ is a subgroup of $N$. Clearly, $1\in N$ and $v^{1\cdot \ell}=v^\ell=v$. Now, let $m_1,m_2\in N$ with $v^{m_1\ell}=v^{m_1}$ and $v^{m_2\ell}=v^{m_2}$. Since $\ell$ normalises $N^i$ we have $(m_2\vert_{\mathcal{O}_i})^\ell\in N^i$ and hence there exists $m_3\in N$ with $m_3\vert_{\mathcal{O}_i}=(m_2\vert_{\mathcal{O}_i})^\ell$. The image of the point $v\in\mathcal{O}_i$ under $(m_2\vert_{\mathcal{O}_i})^\ell$ is $v^{\ell^{-1}m_2\ell }=v^{m_2\ell}=v^{m_2}$. Hence $v^{m_3}=v^{m_2}$ and, since $N$ acts regularly on $\mathcal{O}_i$, we get $m_3=m_2$.  Therefore $m_2\vert_{\mathcal{O}_i}=(m_2\vert_{\mathcal{O}_i})^\ell$ and hence $$v^{m_1m_2\ell}=v^{m_1\ell (m_2)^\ell}=(v^{m_1\ell})^{m_2^\ell}=(v^{m_1})^{m_2}=v^{m_1m_2}.$$

The previous paragraph shows that, if $\ell\in L^i\setminus\{1\}$ fixes $v\in \mathcal{O}_i$, then $\{m\in N\mid v^{m\ell}=v^n\}$ is a proper subgroup of $N$ and hence $|\{w\in \mathcal{O}_i\mid w^\ell=w\}|=|\{v^m\mid m\in N,v^{m\ell}=v^m\}|=|\{m\in N\mid v^{m\ell}=v^m\}|\le |N|/2$. Thus $\ell$ fixes at most $n/2$ elements of $\mathcal{O}_i$. 
Therefore the number of subsets $S_i$ of $\mathcal{O}_i$ invariant under $\ell\in L^i\setminus\{1\}$ is at most $2^{3n/4}$. 

For $S\in \Phi_i$, observe that $S_j$ is an arbitrary subset of $\mathcal{O}_j$ when $j\neq i$; moreover, $(F_S)\mid_{\mathcal{O}_i}\leq L^i$. This shows that maps $f$ that satisfy the conditions described in the definition of $\Phi_i$ lie in $L^i$, and since they are also in $F_S \le \Aut(\Gamma)$, they must fix $S_i$ setwise but not pointwise. So we can bound the number of choices for $S_i$ by the number of choices for such an $f$ times the number of subsets of $\mathcal O_i$ that are fixed by that $f$. This gives us at most $|L^i|2^{3n/4}$ choices for $S_i$. Therefore we
get  
$$|\Phi_i|\leq |L^i|2^{3n/4}(2^{n})^{b-1}\le n^{\log_2(n)+1}2^{r-n/4},$$
and the lemma follows.
\end{proof}

\begin{lemma}{{{\rm (See~\cite[Lemma $4.2$]{BaGo}.)}}}\label{lemma4242} For each $i\in \{2,\ldots,b\}$, 
$$|\{S\subseteq R\mid \textrm{there exists }f\in F_S \textrm{ with }v_0^f=v_0,\textrm{ and }f\vert_{\mathcal{O}_i}\neq 1\}|\leq 2^r{n\choose 2}\left(\frac{3}{4}\right)^{\frac{r/n-2}{3}}.$$
\end{lemma}
\begin{proof}
Fix $i\in \{2,\ldots,r\}$ and denote by $\Phi_i$ the set $$\Phi_i:=\{S\subseteq R\mid \textrm{there exists }f\in (F_S)_{v_0} \textrm{ with }f\vert_{\mathcal{O}_i}\neq 1\}.$$

Let $S$ be a subset of $R$. For a vertex $u$ of $\Cay(R,S)$ in $\mathcal{O}_i$, let $\sigma(S,u,j)$ denote  the outneighbours of $v_0$ and  $u$ lying in $\mathcal{O}_j$. Given $g_u\in R$ with $v_0^{g_u}=u$, it is clear that
$$\sigma(S,u,j)=S\cap S^{g_u}\cap \mathcal{O}_j=S_j\cap S^{g_u}.$$

Let $s\in S$ with $s^{g_u}\in S_j$. Then $s^{g_u}\in \mathcal{O}_j=v_0^{\gamma_jN}=v_0^{N\gamma_j}$ and $s^{g_u\gamma_j^{-1}}\in v_0^{N}=\mathcal{O}_1$. Since $g_u$ maps the element $v_0$ of $\mathcal{O}_1$ to the element $u$ of $\mathcal{O}_i$, we see that $g_u\in \gamma_iN$ and $s\in \mathcal{O}_1^{\gamma_jg_u^{-1}}=v_0^{\gamma_j\gamma_i^{-1}N}=\mathcal{O}_{ji^{-1}}$. This shows that 
\begin{equation}\label{eq6-}
\sigma(S,u,j)=S_j\cap S_{ji^{-1}}^{g_u}.
\end{equation}

For two distinct vertices $u,v\in \mathcal{O}_i$, let $$\Psi_i(u,v,j):=\{S\subseteq R\mid |\sigma(S,u,j)|\equiv |\sigma(S,v,j)|\mod 2\}.$$
We claim that 
\begin{equation}\label{eq6+}
|\Psi_i(u,v,j)|\leq	\frac{3}{4}\cdot 2^{r}.
\end{equation}
Since $u,v\in\mathcal{O}_i$, we have $u=v_0^{\gamma_in_u}$ and $v=v_0^{\gamma_in_v}$, for some $n_u,n_v\in N$. Let $S\in \Psi_i(u,v,j)$. From~\eqref{eq6-}, we obtain  
\begin{equation}\label{eq6++}
|\sigma(S,u,j)|=|S_{ji^{-1}}\cap S_j^{n_u^{-1}\gamma_i^{-1}}|\quad\textrm{and}\quad |\sigma(S,v,j)|=|S_{ji^{-1}}\cap S_j^{n_v^{-1}\gamma_i^{-1}}|.
\end{equation}
From this we see that the condition $|\sigma(S,u,j)|\equiv|\sigma(S,v,j)|\mod 2$ does not impose any constraints on $S_k$, for $k\notin \{j,ji^{-1}\}$. Therefore 
$$|\Psi_i(u,v,j)|= A\cdot 2^{r-2n},$$
where $A$ is the number of pairs of subsets $S_{ji^{-1}}\subseteq \mathcal{O}_{ji^{-1}}$ and $S_j\subseteq\mathcal{O}_j$ with 
\begin{equation}\label{label}
|S_{ji^{-1}}\cap S_j^{n_u^{-1}\gamma_i^{-1}}|\equiv|S_{ji^{-1}}\cap S_j^{n_v^{-1}\gamma_i^{-1}}|\mod 2.
\end{equation}

Let $x$ be the number of subsets $S_j$ of $\mathcal{O}_j$ with $S_j^{n_u^{-1}}=S_j^{n_v^{-1}}$, and let $y=2^n-x$.

Observe that for every subset $S\subseteq R$ with $S_j^{n_u^{-1}}=S_j^{n_v^{-1}}$, we have $S\in \Psi_i(u,v,j)$ because \eqref{label} is automatically satisfied in this case. Now $n_v^{-1}n_u\in N\setminus\{1\}$ and, if $S_j=S_j^{n_v^{-1}n_u}$, then $S_j$ is a union of $\langle n_v^{-1}n_u\rangle$-cosets. As $o(n_v^{-1}n_u)\geq 2$ and as $N$ acts regularly on $\mathcal{O}_j$, we have $x\leq 2^{n/2}$. 

Next let $S\in \Psi_i(u,v,j)$  and suppose that $S_j$ is a subset of $\mathcal{O}_j$ with $S_j^{n_u^{-1}}\neq S_j^{n_v^{-1}}$. Now $S_j^{n_u^{-1}\gamma_i^{-1}}$ and $S_j^{n_v^{-1}\gamma_i^{-1}}$ are two distinct subsets of $\mathcal{O}_{ji^{-1}}$ of the same size $a$, say. Let $b$ be the size of $S_j^{n_u^{-1}\gamma_i^{-1}}\cap S_j^{n_v^{-1}\gamma_i^{-1}}$. Observe that $a-b>0$. Moreover,  a subset $S_{ji^{-1}}$ of $\mathcal{O}_{ji^{-1}}$ with $|S_{ji^{-1}}\cap S_j^{n_u^{-1}\gamma_i^{-1}}|\equiv |S_{ji^{-1}}\cap S_j^{n_v^{-1}\gamma_i^{-1}}|\mod 2$ can be written as $X\cup Y$, where $X$ is as an arbitrary subset of $\mathcal{O}_{ji^{-1}}\setminus (S_j^{n_v^{-1}\gamma_i^{-1}}\setminus S_j^{n_u^{-1}\gamma_i^{-1}})$ and $Y$ is a subset of $S_j^{n_v^{-1}\gamma_i^{-1}}\setminus S_j^{n_u^{-1}\gamma_i^{-1}}$ of size having parity uniquely determined by the parity of $|X|$. Therefore we have $2^{n-(a-b)}2^{(a-b)-1}=2^{n-1}$ choices for $S_{ji^{-1}}$. Altogether we have 
\begin{eqnarray*}
A&=& x\cdot 2^n+y\cdot 2^{n-1}=x2^n+2^{2n-1}-x2^{n-1}=2^{2n-1}+x2^{n-1}\\
&\leq& 2^{2n-1}+2^{n/2}2^{n-1}=\left(\frac{1}{2}+\frac{1}{2^{n/2+1}}\right)2^{2n}\leq \frac{3}{4}\cdot 2^{2n}
\end{eqnarray*} and~\eqref{eq6+} is proved. 

\smallskip

Choose a subset $J\subseteq \{2,\ldots,b\}\setminus\{i\}$ of maximal size with $J\cap Ji^{-1}=\emptyset$. We claim that $|J|\geq (b-2)/3$. We argue by contradiction and we suppose that $|J|<(b-2)/3$. Clearly $|J\cup Ji\cup Ji^{-1}|\le |J|+|Ji|+|Ji^{-1}|=3|J|<b-2$  and hence there exists $x\in \{1,\ldots,b\}\setminus(J\cup Ji\cup Ji^{-1}\cup\{1,i\})$. Set $J'=J\cup\{x\}$ and observe that $|J'|=|J|+1$ and $J'\subseteq \{2,\ldots,b\}\setminus\{i\}$. Since $J\cap Ji^{-1}=\emptyset$ and $i\neq 1$, we have $J'\cap J'i^{-1}=(\{x\}\cap Ji^{-1})\cup(J\cap\{xi^{-1}\})=\emptyset$.
 
Given two distinct elements $u,v$ in $\mathcal{O}_i$, define $\Psi_i(u,v,J):=\bigcap_{j\in J}\Psi_i(u,v,j)$. Observe that from~\eqref{eq6++} and from the definition of $J$ (requiring that $J \cap Ji^{-1}=\emptyset$) the events $\{\Psi_i(u,v,j)\}_{j\in J}$ are pairwise independent. Therefore it follows from~\eqref{eq6+} that 
$$|\Psi_i(u,v,J)|\leq \left(\frac{3}{4}\right)^{|J|}2^r.$$

We are now ready to conclude the proof of this lemma. Let $S\in \Phi_i$ and let $f\in (F_S)_{v_0}$ with $f\vert_{\mathcal{O}_i}\neq 1$. Let $u$ and $v$ be distinct vertices of $\Cay(R,S)$ in $\mathcal{O}_i$ with $u^f=v$. Since $f$ fixes $v_0$ and fixes every $N$-orbit setwise, we get $(\sigma(S,u,j))^f=\sigma(S,u^f,j)=\sigma(S,v,j)$, for every $j\in \{2,\ldots,b\}$ and hence (in particular) $S\in \Psi_i(u,v,J)$. Therefore, given $u$ and $v$, we have at most $|\Psi_i(u,v,J)|\leq (3/4)^{(b-2)/3}2^r$ choices for $S$. As we have ${n\choose 2}$ choices for $\{u,v\}$, we have 

\begin{eqnarray*}
|\Phi_i|&\leq& {n\choose 2}\left(\frac{3}{4}\right)^{\frac{b-2}{3}}2^r
\end{eqnarray*}
and the lemma follows.
\end{proof}

We are now ready to state the first reduction.

\begin{theorem}\label{red1}
Let $R$ be a finite group of order $r$ and let $n$ be a positive integer with $n\ge 71$. The number of subsets $S$ of $R$ such that there exists
\begin{itemize}
\item  a non-identity proper normal subgroup $N$ of $R$ with $|N|\ge n$ and 
\item an automorphism $f\in \Aut(\Cay(R,S))$ normalising $N$, with $f\notin R$ and with $f$ fixing setwise every $N$-orbit 
\end{itemize}
is at most $2^{r-\frac{n}{4}+(\log_2(n))^2+(\log_2(r))^2+\log_2(r)}.$
\end{theorem}

\begin{proof}
Every subgroup of $R$ has at most $\log_2(r)$ generators and hence $R$ has at most $r^{\log_2(r)}=2^{(\log_2(r))^2}$ subgroups. In particular, we have at most $2^{(\log_2(r))^2}$ choices for a non-identity proper normal subgroup $N$ of $R$. Now, fix such a normal subgroup  $N$, and let $S\subseteq R$ such that there exists $g\in \Aut(\Cay(R,S))\setminus R$ normalising $N$ and fixing setwise every $N$-orbit. Thus $g\in F_S\setminus N$. Moreover, replacing $g$ by $gx^{-1}$, for a suitable $x\in N$ if necessary, we may assume that $g$ fixes the vertex $v_0\in \mathcal{O}_1$, that is, $g\in (F_S)_{v_0}\setminus \{1\}$. 

Since our original $g$ was not in $R$, we certainly have $g \neq 1$, so there exists $i\in \{2,\ldots,b\}$ such that $g\vert_{\mathcal{O}_i}\neq 1$, where $g\in (F_S)_{v_0}$.  By Lemma~\ref{lemma42}, we have at most $(|R:N|-1)M'$ choices for $S$, where  $$M'=2^{r-\frac{|N|}{4}+(\log_2(|N|))^2+\log_2(|N|)}$$ (observe that the factor $|R:N|-1$ counts  the number of choices of $i$). 
Since $|R:N|=2^{\log_2(r)-\log_2(|N|)}$, this proves that the number of choices for $S$ is at most
$$2^{(\log_2(r))^2}\cdot 2^{\log_2(r)-\log_2(|N|)}\cdot M'.$$
Finally, observe that the mapping $x\mapsto -x/4+(\log_2(x))^2$ is decreasing for $x\ge 71$. Therefore, the lemma follows observing that $|N|\ge n\ge 71$.
\end{proof}

\begin{theorem}\label{red1red1}
Let $R$ be a finite group of order $r$ and let $n$ be a positive integer. The number of subsets $S$ of $R$ such that there exists
\begin{itemize}
\item  a non-identity proper normal subgroup $N$ of $R$ with $|N|\le n$ and 
\item an automorphism $f\in \Aut(\Cay(R,S))$ with $f\notin R$ and with $f$ fixing setwise every $N$-orbit 
\end{itemize}
is at most $2^{r-\frac{r/n-2}{3}\log_2(4/3)+(\log_2(r))^2+\log_2(r)+\log_2(n)-1}.$
\end{theorem}
\begin{proof}The proof follows verbatim   the proof of Theorem \ref{red1} replacing Lemma \ref{lemma42} with Lemma \ref{lemma4242} and noticing that $|N|\le n$  in this case. 
\end{proof}

It is important to observe that in Theorem \ref{red1} we require $N$ to be normalised by $f$ and not too small,  whereas  in Theorem \ref{red1red1} we do not require $N$ to be normalised by $f$ however we do require $N$ to be small.

\section{Preliminary lemmas and second reduction}\label{sec2}
In this section, as usual, we let $R$ be a finite group of order $r$ and we represent $R$ as a regular subgroup of $\Sym(r)$. We show that the problem of enumerating Cayley digraphs over $R$ is strictly related (but possibly not equivalent) to the problem of enumerating the subgroups $G$ of $\Sym(r)$ with $R<G$ and with $R$ maximal in $G$. Next, we show that the number of such groups $G$ with $|G|$ ``small'' is negligible and we will deduce yet another  useful reduction for the problem of asymptotically enumerating Cayley digraphs.

\begin{lemma}\label{lemma1}
Let $G$ be a transitive subgroup of $\Sym(\Omega)$, let $\omega\in \Omega$  and let $\kappa$ be the number of orbits of $G_\omega$ on $\Omega$. Then there exist   $2^{\kappa}$  digraphs $\Gamma$ with $\Omega=V\Gamma$ and  $G\leq \Aut(\Gamma)$. Moreover, if $G$ is not regular, then $\kappa\le\frac{3}{4}|\Omega|$.
\end{lemma}
\begin{proof}
Since $G$ is acting transitively on $\Omega$, a digraph $\Gamma$ with vertex set $\Omega$ and with $G\le \Aut(\Gamma)$ is uniquely determined by the out-neighbourhood $\Gamma^+(\omega)$ of the given vertex $\omega$. As $\Gamma^+(\omega)$ is a union of $G_\omega$-orbits, we have $2^\kappa$ choices for $\Gamma^+(\omega)$ and hence $2^\kappa$ choices for $\Gamma$.

Suppose now that $G$ is not regular on $\Omega$. Set $\Delta:=\{\delta\in \Omega\mid G_\omega \textrm{ fixes }\delta\}$. Since $\Delta$ is a block for  a system of imprimitivity for $G$, $|\Delta|$ divides $|\Omega|$. Since $G$ is not regular, we have $G_\omega\neq 1$ and hence $|\Delta|<|\Omega|$, so $|\Delta|\leq |\Omega|/2$. Clearly, $G_\omega$ has at most $(|\Omega|-|\Delta|)/2$ orbits on $\Omega\setminus \Delta$. Thus $G_\omega$ has at most $$
|\Delta|+\frac{|\Omega|-|\Delta|}{2}=
\frac{1}{2}|\Delta|+\frac{|\Omega|}{2}\leq 
\frac{|\Omega|}{4}+\frac{|\Omega|}{2}=\frac{3|\Omega|}{4}$$
orbits on $\Omega$. (We note that this is the same argument used in Remark~\ref{rem : 1}(3).)
\end{proof}

\begin{lemma}\label{lemma2}
For every $\varepsilon\in (0,1/2)$, there exists $r_\varepsilon\in\mathbb{N}$ such that, for every $r>r_{\varepsilon}$ and for every regular subgroup $R$ of $\Sym(r)$, the number of subgroups $G$ of $\Sym(r)$ with
\begin{itemize}
\item $R<G$, 
\item $G$ having at most $\log_2(r)+1$ generators and
\item $|G_1|\leq 2^{r^{1/2-\varepsilon}}$,
\end{itemize} is at most $2^{r^{1-\varepsilon}}$.
\end{lemma}
\begin{proof}
Fix $\varepsilon\in (0,1/2)$.  We let $r_{\varepsilon}\in \mathbb{N}$ be the smallest positive integer such that
\begin{equation}\label{eq0}
(2\log_2(r)+5)r^{1-2\varepsilon}+
(\log_2(r)+1)r^{1/2-\varepsilon}+
2(\log_2(r))^2\leq r^{1-\varepsilon},
\end{equation}
for every $r\geq r_{\varepsilon}$.  Comparing the asymptotics of the right-hand side and the left-hand side of~\eqref{eq0}, we see that $r_{\varepsilon}$ is well-defined.


Given $G$ and $G'$ two  abstract groups and $H\leq G$, $H'\leq G'$, we write $(G,H)\sim (G',H')$ if there exists a group isomorphism $\varphi:G\to G'$ with $H^\varphi=H'$. Clearly, $\sim$ defines an equivalence relation. We denote by $[G,H]$ the $\sim$-equivalence class containing $(G,H)$. Now consider 
\begin{eqnarray*}
\mathcal{M}=\{[G,H]&\mid& G \textrm{ is a group}, H\leq G, |G|\leq 2^{r^{1/2-\varepsilon}}\\
&& \textrm{ and }G \textrm{ is }(\log_2(r)+1)\textrm{-generated}\}.
\end{eqnarray*} 

\smallskip

\noindent\textsc{Claim 1: }We have
\begin{equation}\label{eq2}
|\mathcal{M}|\leq 2^{(2\log_2(r)+5)r^{1-2\varepsilon}+r^{1/2-\varepsilon}}.
\end{equation}

\smallskip

  \noindent  From~\cite[Theorem~$1$]{Lub}  together with~\cite[Remark~$3$(1)]{Lub} we get that the number of isomorphism classes of groups of order $N$ that are $d$-generated is at most $N^{2(d+1)\log_2(N)}=2^{2(d+1)(\log_2(|N|))^2}$. In particular, applying this theorem with $d:=\log_2(r)+1$ and with $N\leq 2^{r^{1/2-\varepsilon}}$, we get that the number of groups $G$ that are $(\log_2(r)+1)$-generated and of order at most $2^{r^{1/2-\varepsilon}}$ is at most $2^{2(\log_2(r)+2)r^{1-2\varepsilon}}\cdot 2^{r^{1/2-\varepsilon}}$ (observe that the second factor counts the number of choices for $N$: the cardinality of $G$).
Now, let $G$ be a group of order at most $2^{r^{1/2-\varepsilon}}$. Since every subgroup of $G$ is at most $\log_2(|G|)$-generated, the number of subgroups $H$ of $G$ is at most $|G|^{\log_2(|G|)}\leq 2^{r^{1-2\varepsilon}}$, and hence our claim is proved.~$_\blacksquare$

\smallskip

Now, let $R$ be a regular subgroup of $\Sym(r)$ and let $\mathcal{S}_R$ be the set of subgroups of $\Sym(r)$ with $R<G$, with $G$ having at most $(\log_2(r)+1)$ generators and with $|G|\leq 2^{r^{1/2-\varepsilon}+\log_2(r)}$.

 \smallskip
 
  \noindent\textsc{Claim 2:}  We have 
\begin{equation}\label{eq1}
|\mathcal{S}_R|\leq 2^{\log_2(r)r^{1/2-\varepsilon}+2(\log_2(r))^2}|\mathcal{M}|.
\end{equation}

\smallskip

\noindent Observe that every element $G$ of $\mathcal{S}_R$ determines an element of $\mathcal{M}$ via the mapping $\varphi:G\mapsto [G,G_1]$ where $G_1$ is the stabiliser of $1$ in $G$. 
We show that there are at most $2^{\log_2(r)r^{1/2-\varepsilon}+2(\log_2(r))^2}$ elements of $\mathcal{S}_R$ having the same image via $\varphi$, from which~\eqref{eq1} immediately follows. We argue by contradiction and we let $G^1,\ldots,G^\ell\in\mathcal{S}_R$ with $\varphi(G^i)=\varphi(G^1)$, for every $i\in \{1,\ldots,\ell\}$, where $\ell>2^{\log_2(r)r^{1/2-\varepsilon}+2(\log_2(r))^2}$. Thus there exists a group isomorphism $\phi_i:G^1\to G^i$ with $(G^i)_1=((G^1)_1)^{\phi_i}$. Therefore the permutation representation of $G^1$ on the coset space $G^1/(G^1)_1$  is permutation isomorphic to the permutation representation of $G^i$ on the coset space $G^i/(G^i)_1$. Thus $G^1$ and $G^i$ are conjugate via an element of $\Sym(r)$, that is, $G^1=(G^i)^{\sigma_i}$ for some $\sigma_i\in \Sym(r)$. Now, as $G^1$ acts transitively on $\{1,\ldots,r\}$, replacing $\sigma_i$ by an element of the form $g_i\sigma_i$ (for some $g_i\in G^1$), we may assume that $\sigma_i$ fixes $1$, that is, $1^{\sigma_i}=1$.

As $R\leq G^i$ for every $i$, we get that $R^{\sigma_1},\ldots,R^{\sigma_\ell}$ are $\ell$ regular subgroups of $G^1$. Since $R$ is $\log_2(r)$-generated, we see that $G^1$ contains at most $|G^1|^{\log_2(r)}\leq 2^{\log_2(r)r^{1/2-\varepsilon}+(\log_2(r))^2}$ distinct subgroups of order $r$. In particular, since $\ell>2^{\log_2(r)r^{1/2-\varepsilon}+2(\log_2(r))^2}$, we see that $R^{\sigma_{i_1}}=\cdots=R^{\sigma_{i_t}}$ for some $t>2^{(\log_2(r))^2}$ and some subset $\{i_1,\ldots,i_t\}$ of size $t$ of $\{1,\ldots,\ell\}$. Therefore $\sigma_{i_1}\sigma_{i_j}^{-1}$ normalises $R$. As $1^{\sigma_{i_1}\sigma_{i_j}^{-1}}=1$, $\sigma_{i_1}\sigma_{i_j}^{-1}$ is an automorphism of $R$, for every $j\in \{1,\ldots,t\}$. Since $R$ has at most $|R|^{\log_2(r)}=2^{(\log_2(r))^2}$ automorphisms, we get $\sigma_{i_1}\sigma_{i_j}^{-1}=\sigma_{i_1}\sigma_{i_{j'}}^{-1}$ for two distinct indices $j$ and $j'$. Thus $\sigma_{i_j}=\sigma_{i_{j'}}$ and $G^{i_j}=(G^1)^{\sigma_{i_j}^{-1}}=(G^1)^{\sigma_{i_{j'}}^{-1}}=G^{i_{j'}}$, which is a contradiction.~$_\blacksquare$

\smallskip

From~\eqref{eq0},~\eqref{eq2} and~\eqref{eq1}, we have $$|\mathcal{S}_R|\le 2^{r^{1-\varepsilon}},$$
that is, the number of subgroups $G$ of $\Sym(r)$ with $R<G$, with $G$ having at most $\log_2(r)+1$ generators and  with $|G|\le 2^{r^{1/2-\varepsilon}+\log_2(r)}$ is at most $2^{r^{1-\varepsilon}}$. Now, whenever $G\le \Sym(r)$ with $R<G$, $G$ has at most $\log_2(r)+1$ generators, and $|G_1|\le 2^{r^{1/2-\varepsilon}}$, we must have
 $$|G|=r|G|_1\le r2^{r^{1/2-\varepsilon}}=2^{r^{1/2-\varepsilon}+\log_2(r)},$$ so that $G \in \mathcal{S}_R$. The proof of this lemma immediately follows.
\end{proof}

We are now ready to give two more  reductions.
\begin{theorem}\label{red2}Let $R$ be a finite group of order $r$. For every $\varepsilon\in (0,1/2)$, there exists $r_\varepsilon$ such that if $r\geq r_\varepsilon$, then the number of subsets $S$ of $R$ such that $\Aut(\Cay(R,S))$ contains a subgroup $G$ with
\begin{itemize}
\item  $R<G$ and 
\item $|G_1|\leq 2^{r^{1/2-\varepsilon}}$,
\end{itemize} is at most $2^{3r/4+r^{1-\varepsilon}}$.
\end{theorem}
\begin{proof}
Given $\varepsilon\in (0,1/2)$, using Lemma~\ref{lemma2} choose $r_\varepsilon$ such that, for $r\geq r_\varepsilon$, the number of subgroups $G$ of $\Sym(r)$ with $R<G$, with $G$ having  at most $\log_2(r)+1$ generators and with $|G_1|\leq 2^{r^{1/2-\varepsilon}}$ is at most $2^{r^{1-\varepsilon}}$.

Let $S_1,\ldots,S_\ell$ be the subsets of $R$ such that $\Aut(\Cay(R,S_i))$ contains a subgroup $G'^i$ with $R<G'^i$ and with $|G_1'^i|\leq 2^{r^{1/2-\varepsilon}}$. We show that $\ell\leq 2^{3r/4+r^{1-\varepsilon}}$. We argue by contradiction and we assume that $\ell>2^{3r/4+r^{1-\varepsilon}}$. For each $i$, fix a subgroup $R<G^i\leq G'^i$ with $R$ maximal in $G^i$. Observe that, since $R$ is at most $\log_2(r)$-generated, $G^i$ is at most $(\log_2(r)+1)$-generated. In particular, by Lemma~\ref{lemma2}, the set $\{G^1,\ldots,G^\ell\}$ contains at most $2^{r^{1-\varepsilon}}$ distinct elements. By the pigeonhole principle, there exists a group $G^{i_0}$ such that $R<G^{i_0}\leq \Aut(\Cay(R,S))$ for more than $2^{3r/4}$ subsets $S$ of $R$. However this contradicts Lemmas~\ref{lemma1}.
\end{proof}

\begin{theorem}\label{rd3}
Let $R$ be a finite group of order $r$. For every $\varepsilon\in (0,1/2)$, there exists $r_\varepsilon$ such that if $r\geq r_\varepsilon$, then the number of subsets $S$ of $R$ such that $\Aut(\Cay(R,S))$ contains a subgroup $G$ with 
\begin{itemize}
\item $R<G$,
\item $R$ maximal in $G$,
\item $|G_1|>2^{r^{1/2-\varepsilon}}$ and
\item the core $G_R:=\bigcap_{g\in G}R^g$ of $R$ in $G$ has size greater than $4\log_2(r)$, 
\end{itemize}is at most $2^{r-\frac{r}{4\log_2(r)}\log_2(e)-\log_2(4\log_2(r))+(\log_2(r))^2+\log_2(r)}$.
\end{theorem}
\begin{proof}As usual, we identify $R$ with its image under the right regular representation in $\Sym(r)$ and, given $G\leq \Sym(r)$ with $R<G$, we denote by $G_R:=\cap_{g\in G}R^g$ the core of $R$ in $G$. For each $\varepsilon\in (0,1/2)$, we consider 
\begin{equation*}
\mathcal{S}_{\varepsilon,r}=\{G\leq \Sym(r)\mid R<G,  |G_1|>2^{r^{1/2-\varepsilon}}, R \textrm{ is maximal in }G, |G_R|> 4\log_2(r)\}.
\end{equation*}
Let $r_\varepsilon\in \mathbb{N}$ such that 
\begin{equation}\label{eq4}
{r^{1/2-\varepsilon}}>(\log_2(r))^2,
\end{equation}
 for every $r\geq r_\varepsilon$.

Let $G\in \mathcal{S}_{\varepsilon,r}$. Since $G_R\lhd G$, we get $G_1\leq \norm G {G_R}$ and hence $G_1$ acts by conjugation as a group of automorphisms on $G_R$. Since $|\Aut(G_R)|\leq |G_R|^{\log_2(|G_R|)}\leq 2^{(\log_2(r))^2}$ and since $r^{1/2-\varepsilon}>(\log_2(r))^2$, there exists $g\in \cent {G_1}{G_R}$ with $g\neq 1$. Since $R$ is maximal in $G$, we get $G=\langle R,g\rangle$ and hence the group $G$ is uniquely determined by a non-identity element $g$ of $\cent {\Sym(r)}{G_R}$. Observe that $\cent {\Sym(r)}{G_R}$ is uniquely determined by the normal subgroup $G_R$ of $R$.

The group $G_R$ is a subgroup of $R$ and since $|R|=r$, we see that we have at most $r^{\log_2(r)}=2^{(\log_2(r))^2}$ choices for $G_R$. Now $\cent {\Sym(r)}{G_R}\cong {G_R}\wr\Sym(|R|/|G_R|)$. Hence we have
\begin{eqnarray}\label{eq5555}
|\mathcal{S}_{\varepsilon,r}|\leq 2^{(\log_2(r))^2}\cdot |G_R|^{|R|/|G_R|}(|R|/|G_R|)!
\end{eqnarray}
 (the first term counts the number of choices of $G_R$ and the second term counts the number of choices of $g$). Using~\eqref{eq4} and~\eqref{eq5555}, the inequality $n!\leq n(n/e)^n$ and $|G_R|> 4\log_2(r)$, we get 

\begin{eqnarray*}
\log_2(|\mathcal{S}_{\varepsilon,r}|)&\leq&
(\log_2(r))^2+\frac{|R|\log_2(|G_R|)}{|G_R|}+
\log_2\left(\frac{|R|}{|G_R|}\right)
+\frac{|R|}{|G_R|}\log_2\left(\frac{|R|}{e|G_R|}\right)\\
&=&
(\log_2(r))^2
+
\log_2
\left(
\frac{|R|}{|G_R|}
\right)
+
\frac{|R|}{|G_R|}
\log_2
\left(
\frac{|R|}{e}
\right)\\
&\leq&
(\log_2(r))^2+
\log_2\left(\frac{r}{4\log_2(r)}\right)+
\frac{r}{4\log_2(r)}\left(\log_2(r)-\log_2(e)\right)\\
&=&\frac{r}{4}-\frac{r}{4\log_2(r)}\log_2(e)-\log_2(4\log_2(r))+(\log_2(r))^2+\log_2(r).
\end{eqnarray*}

Now the proof follows by using the last part of the argument in Theorem~\ref{red2}. In fact from Lemma~\ref{lemma1}, for each $G\in \mathcal{S}_{\varepsilon,r}$, there exist at most $2^{3r/4}$ subsets $S$ of $R$ with $G\leq \Aut(\Cay(R,S))$.
\end{proof}

\section{Some notation}\label{sec44}

Let $R$ be a finite regular subgroup of $\Sym(r)=\Sym(\{1,\ldots,r\})$. In the rest of this paper, 
\begin{itemize}
\item we take  $\varepsilon:=0.001,$
\item we choose $r_{\varepsilon}\in\mathbb{N}$ satisfying both Theorems~\ref{red2} and~\ref{rd3} for this choice of $\varepsilon$ and,
\item  we assume that our regular subgroup $R$ satisfies
$r/(4\log_2(r))\ge r_\varepsilon,$ where $r=|R|$.
\end{itemize}
 Since we are interested in the asymptotic number of DRRs, the actual value of $r_\varepsilon$ is not relevant in our arguments. However, with some rough estimates one might show that $r_\varepsilon\le 2^{15\,000}$. 
 
In the light of Theorems~\ref{red2} and~\ref{rd3}, since the number of subsets of $R$ satisfying the hypothesis of either Theorem~\ref {red2} or~\ref{rd3} are negligible compared to $2^r$ when $r$ tends to infinity, we are left with estimating the number of subsets $S$ of $R$ with the property that
\begin{enumerate}
\myitem{(H1)}\label{hyp1f} $\Aut(\Cay(R,S))>R$,
\myitem{(H2strong)}\label{hyp2f} for every subgroup $G$ of $\Aut(\Cay(R,S))$ with $R<G$, the stabiliser $G_1$ has cardinality greater than $2^{r^{0.499}}$,
\myitem{(H3strong)}\label{hyp3f} for every subgroup $G$ as above, the core $G_R:=\bigcap_{g\in G}R^g$ of $R$ in $G$ has cardinality at most $4\log_2(r)$.
\end{enumerate}
First of all, we remark that $G_RG_1$ is the setwise stabiliser of $G_R$ in $G$, where $G_R$ is viewed as the  subset $1^{G_R}=\{1^x\mid x\in G_R\}=G_R$ of the vertex set of $\Cay(R,S)$.

Suppose now that $S\subseteq R$ satisfies \ref{hyp1f}, \ref{hyp2f}, and \ref{hyp3f} and, for some  $G\le\Aut(\Cay(R,S))$ with $R<G$, the subgroup  $G_1$  fixes every $G_R$-orbit setwise. In particular, since $|G_R|$ is ``small'', that is, $|G_R|\le 4\log_2(r)$, Theorem~\ref{red1red1} applied to the normal subgroup $G_R$ of $R$ gives an upper bound on the number of these subsets $S$ of $R$; namely we have at most
\begin{equation}\label{aply4log2}
2^{r-\frac{\frac{r}{4\log_2(r)}-2}{3}\log_2(4/3)+(\log_2(r))^2+\log_2(r)+\log_2(4\log_2(r))-1}.
\end{equation}
choices for $S$. Therefore, since~\eqref{aply4log2} is negligible compared to $2^r$ when $r$ tends to infinity, we only need to estimate the number of subsets $S$ of $R$ which also satisfy the additional property that
\begin{enumerate}
\myitem{(H4strong)}\label{hyp4f} for every subgroup $G$ and $G_R$ as above, some $G_R$-orbit is not fixed (setwise) by $G_1$. In particular, the group $G_RG_1$ is not normal in $G$.
\end{enumerate}
In particular, we need to show that the number of subsets $S$ of $R$ satisfying~\ref{hyp1f},~\ref{hyp2f},~\ref{hyp3f}, and~\ref{hyp4f} is negligible compared to $2^r$. 

At some point, our proof relies on previous cases of our proof, and to make that argument easier it is much more convenient to work under weaker hypotheses. Therefore, we are interested in the subsets $S$ of $R$ with the property that
\begin{enumerate}
\myitem{(H1)}\label{hyp1} $\Aut(\Cay(R,S))>R$, and \textbf{{\em for some}} subgroup $G$ of $\Aut(\Cay(R,S))$, we have
\myitem{(H2)}\label{hyp2}$R$ is maximal in $G$ and the stabiliser $G_1$ has cardinality greater than $2^{r^{0.499}}$,
\myitem{(H3)}\label{hyp3} the core $G_R:=\bigcap_{g\in G}R^g$ of $R$ in $G$ has cardinality at most $4\log_2(r)$,
\myitem{(H4)}\label{hyp4} some $G_R$-orbit is not fixed (setwise) by $G_1$. In particular, the group $G_RG_1$ is not normal.
\end{enumerate}
Observe that if a subset $S$ of $R$ satisfies~\ref{hyp1f},~\ref{hyp2f},~\ref{hyp3f},~\ref{hyp4f}, then $S$ satisfies also~\ref{hyp1},~\ref{hyp2},~\ref{hyp3} and~\ref{hyp4}, so if we can show that the number of sets satisfying the weaker hypotheses is negligible compared to $2^r$, this will be sufficient for our purposes.

In what follows, we also need the reduction given by Theorem~\ref{red1}, but since its role will appear only later in our work we do not include it here in our notation.

\begin{definition}\label{def1is}{\rm We denote by $\mathcal{T}$  the subsets of $R$ satisfying~\ref{hyp1}--\ref{hyp4}. The set $\mathcal{T}$ depends upon the group $R$ and hence, in principal, we need a notation depending on $R$, however we find that this would make our notation too cumbersome to use. 

Moreover, we denote by $\mathcal{T}'$ the elements $S\in\mathcal T$ such that there exists $G\le \Aut(\Cay(R,S))$ satisfying~\ref{hyp2},~\ref{hyp3} and~\ref{hyp4} and with
\begin{enumerate}
\myitem{(H5)}\label{hyp5}  $G_R=\bigcap_{g\in G}R^g=1$.
\end{enumerate} 
For each $S\in \mathcal{T}'$, choose once and for all  $G_S\le\Aut(\Cay(R,S))$ witnessing that $S$ does belong to $\mathcal{T}'$.
 In particular, $G_S$ depends upon the set $S$.}
\end{definition}

\begin{lemma}\label{lemma3sec4is}
For each $S\in\mathcal{T}'$, the group $G_S$ acts primitively and faithfully on the set of  right cosets  of $R$ in $G_S$ and $(G_S)_1$ is a non-normal regular subgroup of $G_S$.
\end{lemma}
\begin{proof}
Write $G:=G_S$. The fact that $G$ acts primitively and faithfully on the set of  right cosets of $R$ in $G$ follows from the maximality of $R$ in $G$ and from $1=G_R=\bigcap_{g\in G}R^g$. From~\ref{hyp4}, $G_1G_R=G_1$ is not normal in $G$. Finally, as $G=G_1R$ and $G_1\cap R=1$, we deduce that $G_1$ acts regularly in this primitive action.
\end{proof}

Definition~\ref{def1is} and Lemma~\ref{lemma3sec4is} set up a natural link between our original problem of enumerating Cayley digraphs and the powerful theory of finite primitive groups.

From now on, for each $S\in \mathcal{T}'$, the group $G_S$ is endowed with two faithful actions: the primitive action on the set of  right cosets of $R$ in $G_S$ and the transitive action on the vertices of the Cayley digraph $\Cay(R,S)$. We try henceforth to emphasise which action of $G_S$ we are considering; this hopefully avoids possible confusion.

The modern key for analysing a finite primitive permutation group $L$ is to
study the \textit{socle} $N$ of $L$, that is, the subgroup generated
by the minimal normal subgroups of $L$. The socle of an arbitrary
finite group is isomorphic to the non-trivial direct product of simple
groups; moreover, for finite primitive groups these simple groups are
pairwise isomorphic. The O'Nan-Scott theorem describes in details the
embedding of $N$ in $L$ and collects some useful information about the
action of $N$. In~\cite[Theorem]{LPSLPS} five types of primitive groups
are defined (depending on the group- and action-structure of the
socle), namely HA (\emph{Affine}), AS (\emph{Almost Simple}), SD
(\emph{Simple Diagonal}), PA (\emph{Product Action}) and TW
(\emph{Twisted Wreath}), and it is shown that  every primitive group
belongs to exactly one of these types. We remark that in~\cite{C3}
this subdivision into types is refined, namely the PA type
in~\cite{LPSLPS} is partitioned in four parts, which are called HS (\emph{Holomorphic simple}), HC (\emph{Holomorphic compound}), CD
(\emph{Compound Diagonal}) and PA.  For what follows it is convenient to use this subdivision into eight types of the finite primitive primitive groups.

\begin{definition}\label{def2}{\rm For each $\mathcal{C}\in \{HA, HS, HC, SD, CD, TW, AS, PA\}$, we let $\mathcal{T}'^{\mathcal{C}}$ be the elements $S\in\mathcal{T}'$ with $G_S$ having O'Nan-Scott type $\mathcal{C}$ in its action on the set $\Omega_S:=R\backslash G_S$ of right cosets of $R$ in $G_S$. Moreover, we let $P_S$ be the socle of $G_S$. Thus, we have the following  partition of $\mathcal{T}'$:
$$\mathcal{T}'=
\mathcal{T}'^{HA}\cup 
\mathcal{T}'^{HS}\cup 
\mathcal{T}'^{HC}\cup 
\mathcal{T}'^{SD}\cup 
\mathcal{T}'^{CD}\cup 
\mathcal{T}'^{TW}\cup 
\mathcal{T}'^{AS}\cup 
\mathcal{T}'^{PA}.
$$}
\end{definition}

In the next section we aim to prove a strong upper bound for the cardinality of $\mathcal{T}'$. Then we use this strong upper bound on $|\mathcal{T}'|$ to obtain a weaker upper bound (but still adequate for our purposes) for $|\mathcal{T}|$.

\section{Estimating the cardinality of  $\mathcal{T}'$}\label{sec:stopit}
Recall, from Definition \ref{def1is} in Section~\ref{sec44}, for each $S\in\mathcal{T}'$, we have chosen a certain subgroup $G_S$ of $\Aut(\Cay(R,S))$ and we have denoted by $P_S$ the socle of $G_S$ in its primitive action on $\Omega_S=R\backslash G_S$.

In this section, we estimate the cardinality of $\mathcal{T}'$ by estimating separately the cardinality of $\mathcal{T}'^{\mathcal{C}}$, for each $\mathcal{C}\in  \{HA, HS, HC, SD, CD, TW, AS, PA\}$. In most of our analysis we use detailed information on the factorisations of the almost simple groups, see~\cite{LPS3}.

\subsection{Estimating the cardinality of $\mathcal{T}'^{AS}$}

\begin{lemma}\label{lemmaASASAS}Let $S\in\mathcal{T}'^{AS}$. If $|R|>(3\cdot 29!)!$, then  one of the following happens for some prime $p$:
\begin{description}
\item[(i)]$G_S=\Sym(\{1,\ldots,p\})$ and $(G_S)_1=\Sym(\{1,\ldots,p-2\})$;
\item[(ii)]$G_S=\Alt(\{1,\ldots,p\})$ and $(G_S)_1=(\Sym(\{1,\ldots,p-2\})\times \Sym(\{p-1,p\}))\cap \Alt(p)$;
\item[(iii)]$G_S=\Sym(\{1,\ldots,p\})$ and $(G_S)_1=(\Sym(\{1,\ldots,p-2\})\times \Sym(\{p-1,p\}))\cap \Alt(p)$;
\item[(iv)]$G_S=\Sym(\{1,\ldots,p\})$ and $(G_S)_1=\Alt(\{1,\ldots,p-2\})\times \Sym(\{p-1,p\})$.
\end{description}
\end{lemma}
\begin{proof}
We consider the actions of $G_S$ on $\Omega_S$ and on the vertices $R$ of $\Cay(R,S)$. Suppose that $r=|R|>(3\cdot 29!)!$. Let $n=|(G_S)_1|=|\Omega_S|$ be the degree of $G_S$ in its action on $\Omega_S$.

 \smallskip
 
   Suppose that $G_S$, seen as a primitive subgroup of $\Sym(\Omega_S)$, contains $\Alt(\Omega_S)$. Then $r=|R|\geq (n-1)!/2$ because $R$  is the stabiliser in $G_S$ of a point of $\Omega_S$. Hence (from~\ref{hyp2}) $|(G_S)_1|\geq 2^{((n-1)!/2)^{0.499}}$. Since $G_S=R(G_S)_1$and $R\cap (G_S)_1=1$, we have $n! \ge |G_S|=|R||(G_S)_1|\ge |(G_S)_1|(n-1)!/2$, so $|(G_S)_1|\leq 2n$. With an easy computation, from $$2n\ge 2^{((n-1)!/2)^{0.499}},$$ we get $n\leq 4$. In particular,  $|R|\le |G_S|\le 4!=24<(3\cdot 29!)!$, which is a contradiction. Thus $G_S$, seen as a primitive subgroup of $\Sym(\Omega_S)$, does not contain $\Alt(\Omega_S)$.

Since $R<\Sym(n)$ and $r=|R|>(3\cdot 29!)!$, we have 
\begin{equation}\label{eq:rbounds}(3\cdot 29!)!<r<n!\end{equation} and  $$n>3\cdot 29!.$$

Since $G_S$ is an almost simple group, we have $P_S\unlhd G_S\leq \Aut(P_S)$ and $P_S$ is a non-abelian simple group. Recall that from Lemma \ref{lemma3sec4is}, $(G_S)_1$ acts regularly on $\Omega_S$. 

Now, the almost simple primitive permutation groups admitting a regular subgroup are classified in~\cite{LPS}. From~\cite[Corollary~$1.2$ and Tables~$16.1$,~$16.2$,~$16.3$]{LPS}, we see that (as $\Alt(n)\nleq G_S$ and $n>3\cdot 29!$) one of the following occurs (this is where we really use the very large lower bound on $|R|$, to avoid all exceptional cases):
\begin{enumerate}
\item $P_S=\PSL_m(q)$, $|(G_S)_1|=(q^m-1)/(q-1)$ and $P_S\cap R$ is the stabiliser of a projective point or of a projective line;
\item $P_S=\PSL_2(q)$, $|(G_S)_1|=q(q-1)/2$ and $P_S\cap R\cong D_{q+1}$;
\item $P_S=\Alt(q)$, $|(G_S)_1|=q(q-1)/2$ and $P_S\cap R\cong \Sym(q-2)$;
\item $P_S=\Alt(p)$, $(G_S)_1$ is  isomorphic to $\Sym(p-2)$ or $\Alt(p-2)\times C_2$, and $|P_S\cap R|=p(p-1)/2$, for some prime $p$;
\item$P_S=\Alt(p+1)$, $(G_S)_1\cong\Sym(p-2)$ or $\Alt(p-2)\times 2$, and $|P_S\cap R|=p(p^2-1)/2$, for some prime $p$;
\item $P_S=\Alt(p^2+1)$, $(G_S)_1\cong\Alt(p^2-2)$ and $P_S\cap R\cong \PSL_2(p^2).2$, for some prime $p\equiv 3\mod 4$.
\end{enumerate}
For each of the first three cases, a direct computation using the order of the non-abelian simple group $P_S$ shows that $$|G_S|\leq |\Aut(P_S)|\leq |P_S\cap R|^4\leq |R|^4=r^4.$$ Now $r^4>|(G_S)_1|>2^{r^{0.499}}$ only if $r\leq 1936$, contradicting~\eqref{eq:rbounds}. 

In the fifth case, we have $|G_S|\leq (p+1)!$ and $|R|\geq |R\cap P_S|\geq p(p^2-1)/2$. Now with a computation we see that $(p+1)!>|G_S|>|(G_S)_1|\ge 2^{r^{0.499}}\ge 2^{(p(p^2-1)/2)^{0.499}}$ only if $p\leq 26$. Therefore, $|R|<|G_S|\le 27!$, which is a contradiction to~\eqref{eq:rbounds}.
Similarly, in the sixth case we have $|G_S|\leq (p^2+1)!$ and $|R|\geq |R\cap P_S|= p^2(p^4-1)$, and  the inequality $(p^2+1)!>2^{(p^2(p^4-1))^{0.499}}$ is never satisfied.

We now consider  the fourth case, that is, $P_S=\Alt(p)$, for some prime $p$, $|P_S\cap R|=p(p-1)/2$ and $(G_S)_1\cong\Sym(p-2)$ or $(G_S)_1\cong \Alt(p-2)\times C_2$. A direct case-by-case analysis yields that the only possibilities for $G_S$, $(G_S)_1$ and $R$ are listed in the statement of this lemma (in cases (ii) and (iii) $(G_S)_1 \cong \Sym(p-2)$).
\end{proof}

\begin{theorem}\label{thrmAS}We have $|\mathcal{T}'^{AS}|\le a $, where  $a:=2^{(3\cdot 29!)!}$.
\end{theorem}
\begin{proof}
If $r=|R|\le (3\cdot 29!)!$, then $|\mathcal{T}'^{AS}|\le 2^{r}\le a$. 
Suppose then $r>(3\cdot 29!)!$. Let $S\in\mathcal{T}'^{AS}$.  From Lemma~\ref{lemmaASASAS}, there are only four possibilities for  $G_S$ and $(G_S)_1$: we have only four possibilities for the permutation group $G_S$  in its action on the set of right cosets of $(G_S)_1$, that is, we have only four possibilities for $G_S$ as a permutation group on the vertices of $\Cay(R,S)$. Now, it is an easy computation to see that for each of these four cases $G_S$ in its action on the right cosets of $(G_S)_1$, that is, on the vertices of $\Cay(R,S)$, has rank at most $7$. Therefore, arguing as in Lemma \ref{lemma1}, there are  at most $2^7$ choices for $S$. Since we have at most four choices for $G_S$ and $(G_S)_1\backslash G_S$, we have $|\mathcal{T}'^{AS}|\le 4\cdot 2^7<a$.
\end{proof}

\subsection{Estimating the cardinality of $\mathcal{T}'^{PA}$}
The upper bound in Theorem \ref{thrm:PA} (as well as the upper bound in Theorem~\ref{thrmAS}) should not be taken too seriously, it simply shows that the set $|\mathcal{T'}^{PA}|$ is bounded above by a constant independent on the cardinality of $R$, which in our opinion is an interesting remark on its own.
\begin{theorem}\label{thrm:PA}We have $|\mathcal{T'}^{PA}|\le 2^b$, where $b:=(442^2)!^6\cdot 6! $.
\end{theorem}
\begin{proof}
Given $S\in\mathcal{T}'^{PA}$, we have $G_S\le H\mathrm{wr}\Sym(\kappa)$ endowed of its natural wreath product action on $\Omega=\Delta^\kappa$, where $H$ is a primitive group of AS type on $\Delta$ and $\kappa \ge 2$. The socle $P_S\cong T^\kappa$, where $T$ is the socle of $H$. Replacing $R$ by a suitable conjugate, we may assume that $R=(G_S)_\omega$ where  $\omega=(\delta,\ldots,\delta)\in \Delta^\kappa=\Omega$ with $\delta\in \Delta$. We have $$R=(G_S)_\omega\ge P_S\cap (G_S)_\omega=(P_S)_\omega=T_\delta^\kappa,$$ with $T_\delta\ne 1$: this last fact is immediate because in a primitive group of PA type, the socle $P_S$ does not act regularly. 

As $(G_S)_1$ acts regularly on $\Omega_S$ and $T_\delta\ne 1$, we see that $(G_S)_1$ contains no simple direct factor of $P_S$. Therefore we are in the position to apply Theorem~1~(i) in~\cite{LPS2} to the primitive group $G_S$ of PA type and to its regular subgroup $(G_S)_1$. From~\cite[Theorem~1~(i)]{LPS2}, we deduce that there exists a transitive core-free subgroup $K$ of $H$ in its action on $\Delta$. Unfortunately, there is not enough information in~\cite{LPS2} to guarantee that $K$ acts regularly on $\Delta$, this will make the rest of this proof longer, but in spirit similar to the proof of the AS case done above.

Since $|R\cap P_S|=|T_\delta|^\kappa$ and since $R$ acts transitively by conjugation on the $\kappa$ simple direct summands of $P_S$, we have $|R|\ge \kappa |T_\delta|^\kappa$. As $|(G_S)_1|=|\Omega_S|=|\Delta^\kappa|=|T/T_\delta|^\kappa$ and $|(G_S)_1|\ge 2^{|R|^{0.499}}$, we deduce the inequality
\begin{equation}\label{eqPAeqPA}
|T/T_\delta|^\kappa\ge 2^{(\kappa|T_\delta|^\kappa)^{0.499}}.
\end{equation}

Since $K$ acts transitively on $\Delta$, from the Frattini argument we obtain the factorisation $$H=KH_\delta.$$ As $H$ acts primitively on $\Delta$, $H_\delta$ is a maximal subgroup of $H$. Among all core-free subgroups of $H$ containing $K$ choose one,  $K'$ say, as large as possible. We now consider two cases depending on whether $K'$ is a maximal subgroup of $H$, or $K'$ is not a maximal subgroup of $H$. Observe that in the second case every maximal subgroup of $H$ containing $K$ must contain also the socle $T$ of $H$.

From \eqref{eqPAeqPA} and the transitivity of $K$ on $\Delta$, we deduce
\begin{equation}\label{eqPAeqPAeqPA}
|K'|\ge |K|\ge |\Delta|=|T/T_\delta|\ge 2^{\frac{1}{\kappa}(\kappa|T_\delta|^\kappa)^{0.499}}.
\end{equation}

\smallskip

\noindent\textsc{Case 1:} Suppose  $K'$ is maximal in $H$.

\smallskip
 
\noindent  In this case,  the action of $H$ on the coset space $K'\backslash H$ is faithful and primitive. Write $n:=|H:K'|$. For the reader's convenience we report a very useful result of Mar\'oti~\cite[Theorem~$1.1$]{maroti1} phrased in terms of our current notation: Consider the primitive action of $H$ on $K'\setminus H$ of degree $n$. Then, one of the following holds:
  \begin{description}
  \item[(i)] there exist three natural numbers $m,k,y$ with $m\ge 5$, $m/2>k\ge 1$, $y\ge 1$,  such that $H$ is a subgroup of the wreath product $\Sym(m)\wr\Sym(y)$ containing $(\Alt(m))^y$, where the action of $\Sym(m)$ is on $k$-subsets of $\{1,\ldots,m\}$ and the wreath product has the product action of degree $n={m\choose k}^y$;
  \item[(ii)]$H$ equals $M_{11}$, $M_{12}$, $M_{23}$ or $M_{24}$ in their $4$-transitive actions;
  \item[(iii)]$|H|\le n\cdot \prod_{i=0}^{\lfloor \log_2(n)\rfloor-1}(n-2^i)<n^{1+\lfloor \log_2(n)\rfloor}$.
  \end{description}

We now combine this detailed information on $H$ and its maximal subgroup $K'$ with~\eqref{eqPAeqPAeqPA}. However, first we make two preliminary observations. First, since $H$ is almost simple, in case {\bf (i)} we have $y=1$. Second, $H=K'H_\delta$ and hence 
\begin{equation}\label{eq_442}
n=|H:K'|\le |H_\delta|=|T_\delta||H_\delta:T_\delta|\le |T_\delta|^2,
\end{equation} where in the last inequality we used some basic information on the cardinality of the outer-automorphism group of a non-abelian simple group (here we are using the fact that $|\Out(T)|\le |T_\delta|$, which can be obtained with a case-by-case analysis using the CFSG).

We are now ready to consider the three possibilities: {\bf (i)}, {\bf (ii)} and  {\bf (iii)}. We start with {\bf (iii)}. From~\eqref{eq_442}, we deduce
\begin{equation}\label{eqPAeqPAPA}
|K'|=|H|/n<n^{\lfloor\log_2(n)\rfloor}\le(|T_\delta|^2)^{\log_2(|T_\delta|^2)}.
\end{equation}
From~\eqref{eqPAeqPAeqPA} and~\eqref{eqPAeqPAPA}, we get
$$(|T_\delta|^2)^{\log_2(|T_\delta|^2)}\ge 2^{\frac{1}{\kappa}(\kappa|T_\delta|^\kappa)^{0.499}}.$$
Now a computation shows that this inequality 
is satisfied only when
\begin{itemize}
\item $\kappa=2$ and $|T_\delta|\le 442$, or
\item  $\kappa=3$ and $|T_\delta|\le 30$, or 
\item $\kappa=4$ and $|T_\delta|\le 9$, or 
\item $\kappa=5$ and $|T_\delta|\le 4$, or
\item  $\kappa=6$ and $|T_\delta|=2$.
\end{itemize}  
In particular, $|T_\delta|\le 422$ and hence $n\le 422^2$ by~\eqref{eq_442}. Since $H$ acts faithfully on the cosets of $K'$ (since $K'$ is core-free in $H$), we have $|H|\le |\Sym(n)|=n!\le (442^2)!$. As $\kappa\le 6$, we have $r=|R|\le |G_S|\le |H|^\kappa\cdot \kappa!\le  ((442^2)!)^6\cdot 6!$ and hence the cardinality of $\mathcal{T}'^{PA}$ is bounded above by $2^b$, where $b:=((442^2)!)^6\cdot 6!.$ 

The proof for Case {\bf (ii)} is entirely similar and actually easier. In fact, $H=T$ because $H$ is a non-abelian simple group; therefore $T_\delta=H_\delta$. Moreover,
\begin{align}\label{eqPAeqPAPA1}
|K'|=
\begin{cases}
720&\textrm{when }H=M_{11},\\
7920&\textrm{when }H=M_{12},\\
443520&\textrm{when }H=M_{23},\\
10200960&\textrm{when }H=M_{24},
\end{cases}
&\quad\textrm{and}\quad
|T_\delta|\ge 
\begin{cases}
660&\textrm{when }H=M_{11},\\
72&\textrm{when }H=M_{12},\\
253&\textrm{when }H=M_{23},\\
168&\textrm{when }H=M_{24}.
\end{cases}
\end{align}
(The bound on $|T_\delta|$ follows with a case by case analysis determining the minimal size of a maximal subgroup $X$ of $H$ with $H=K'X$.)
With a computation we see that there is no  solution with $\kappa\ge 2$ of~\eqref{eqPAeqPAeqPA} and~\eqref{eqPAeqPAPA1}.  Therefore,  $\mathcal{T}'^{PA}=\emptyset$ in this case. 

Summing up, we have proved that $\Alt(m)\le H\le \Sym(m)$ and $K'$ is the setwise stabilizer of a $k$-subset of $\{1,\ldots,m\}$ with $1\le k<m/2$. Since $H=K'H_\delta$, we deduce that $H_\delta$ is a $k$-homogeneous group, that is, $H_\delta$ acts transitively on the $k$-subsets of $\{1,\ldots,m\}$. In this concrete action, we have $$|T_\delta|\geq{m\choose k}\quad \textrm{ and }\quad|K'|\le k!(m-k)!$$ and hence~\eqref{eqPAeqPAeqPA} gives
\begin{equation}\label{dagger}
k!(m-k)!\ge 2^{\frac{1}{\kappa}\left({m\choose k}^\kappa\kappa\right)^{0.499}}.
\end{equation}
Observe that the left hand side is at most $m!\le m^m=2^{m\log_2(m)}$ and that ${m\choose k}\ge m$. Recall also that $\kappa\ge 2$. From this and a computation, we obtain that \eqref{dagger} holds true only when
\begin{itemize}
\item $\kappa=2$ and $k=1$, or
\item $\kappa=2$, $k=2$ and $m\le 10$,  or
\item $\kappa=3$, $k=1$, and $m\le 170$.
\end{itemize}
In the last two possibilities, we have $|R|<|G_S|\le |\Sym(m)\mathrm{wr}\Sym(3)|\le 170!^3\cdot 6$ and hence
$|\mathcal{T'}^{PA}|
\leq 
2^{(170!)^3\cdot 6}<2^b$. 

Assume then $k=1$ and $\kappa=2$, that is, $\Alt(m)\mathrm{wr}\Sym(2)\le G_S\le \Sym(m)\mathrm{wr} \Sym(2)$, and $K'$ equals $\Alt(\{1,\ldots,m-1\})$ when $H=\Alt(m)$ or $\Sym(\{1,\ldots,m-1\})$ when $H=\Sym(m)$. Since $H=K'H_\delta$, we deduce that $H_\delta$ is a transitive subgroup of $\Sym(m)$ in its natural action on $\{1,\ldots,m\}$. From this and from the maximality of $H_\delta$ in $H$, it is not difficult to deduce that $T_\delta$ is a transitive subgroup of $\Alt(m)$ in its natural action on $\{1,\ldots,m\}$.

Without loss of generality, we may assume that $m\ge 442^2$,  because otherwise  we again have $|\mathcal{T'}^{PA}|\le 2^b$ (since $|\mathcal T'^{PA}| \le 2^{|R|}$, and $|R|\le |G_S|$). 

With a computation we see that, if $|R|>m^2(m-1)/2-1$, then the inequality $$2(m!)^2\ge |G_S|=|R||(G_S)_1|\ge |R|2^{|R|^{0.499}}$$ is not satisfied when $m\ge 450$. Thus $$|R|\le m^2(m-1)/2-1$$ and hence \begin{equation}\label{eq:boundindex}|G_S:(G_S)_1|=|R|\le m^2(m-1)/2-1.\end{equation} 

For $i\in \{1,2\}$, we let  $\pi_i:P_S\cap (G_S)_1\to \Alt(m)$ be the projection of $P_S\cap(G_S)_1$ in the $i^{\textrm{th}}$ coordinate. Moreover, we let $C_1$ and $C_2$ be the image of $\pi_1$ and $\pi_2$, respectively. Suppose that, for some $i\in \{1,2\}$, $\pi_i$ is surjective. To simplify the notation we assume that $i=1$. Now, $\mathrm{Ker}(\pi_2)$  is normal in $P_S\cap(G_S)_1$ and hence it is normalized by $\pi_1(P_S\cap(G_S)_1)=\Alt(m)$. Therefore, either $\mathrm{Ker}(\pi_2)=\Alt(m)$ or $\pi_2$ is injective. In the first case, we have $P_S\cap (G_S)_1=\Alt(m)\times \Alt(m)$, but this contradicts the fact that $(G_S)_1\cap R=1$. Thus $\pi_2$ is injective and, since $\pi_1$ is surjective, we deduce that $P_S\cap (G_S)_1$ is a diagonal sugroup of $\Alt(m)\times\Alt(m)$. Again this contradicts $(G_S)_1\cap R=1$. So far we have shown that $\pi_1$ and $\pi_2$ are not surjective, that is, $C_1$  and $C_2$ are proper sugroups of $\Alt(m)$. 

Thus $(G_S)_1\le C_1\times C_2$ and hence 
\begin{equation}\label{maslova}
|P_S:C_1\times C_2|\le |P_S:P_S\cap(G_S)_1|=|P_S(G_S)_1:(G_S)_1|\le |G_S:(G_S)_1| \le m^2(m-1)/2-1 \text{ (by~}\eqref{eq:boundindex}).
\end{equation} Clearly, $|\Alt(m):C_1|,|\Alt(m):C_2|\ge m$. If $|\Alt(m):C_i|\ge m(m-1)/2$ for some $i\in \{1,2\}$, then $$|P_S:C_1\times C_2|=|\Alt(m):C_1||\Alt(m):C_2|\ge m^2(m-1)/2$$ contradicting \eqref{maslova}. Therefore $|\Alt(m):C_i|<m(m-1)/2$ for every $i\in \{1,2\}$. Now the only proper subgroup of $\Alt(m)$ having index less then $m(m-1)/2$ is $\Alt(m-1)$, see \cite[Theorem 5.2A]{dixonmortimer} for instance. Therefore $C_1=C_2=\Alt(m-1)$. Since $P_S\cap (G_S)_1$ projects to $\Alt(m-1)$ on both coordinates with a simple argument (using the fact that $\Alt(m-1)$ is simple and \eqref{maslova}), we deduce 
\begin{equation}\label{thatit}
P_S\cap (G_S)_1=\Alt(\{1,\ldots,m-1\})\times \Alt(\{1,\ldots,m-1\}).
\end{equation} We now show that this contradicts the maximality of $R$. Indeed, recall that $T_\delta$ is a transitive subgroup of $\Alt(m)$ in its natural action on $\{1,\ldots,m\}$. Since $R\cap (G_S)_1=1$, from~\eqref{thatit} we deduce $T_\delta\cap \Alt(\{1,\ldots,m-1\})=1$, that is, $T_\delta$ acts regularly on $\{1,\ldots,m\}$. In particular, $T_\delta$ is not a maximal subgroup of $\Alt(m)$ (recall $m>3$). This immediately implies $H=\Sym(m)$, because $H_\delta$ is maximal in $H$. Now, $T_\delta\unlhd H_\delta$ and hence the maximality of $H_\delta$ yields $H_\delta=\norm {H}{T_\delta}$. Since $T_\delta$ is a regular subgroup of $\Sym(m)=H$, we get that $\norm H{T_\delta}$ is the holomorph of $T_\delta$ and hence $|H_\delta|=|T_\delta||\Aut(T_\delta)|$. On the other hand, $|H_\delta|=|H_\delta:T_\delta||T_\delta|=2|T_\delta|$ and hence $|\Aut(T_\delta)|=2$. However, this implies that $T_\delta$ is cyclic of order $3$, which is a contradiction.

\medskip

\noindent\textsc{Case 2: }Suppose  $K'$ is not maximal in $H$, that is,  every maximal subgroup of $H$ containing $K$ must contain also the socle $T$ of $H$. 
\smallskip

\noindent Using the terminology in~\cite{LPS4}, we have $K'\max^-H$, $H_\delta\max^+H$ and $H=K'H_\delta$. Applying~\cite[Theorem]{LPS4} to this factorization, we see that either $K'T=K'(H_\delta \cap K'T)$ is a factorization of the almost simple group $K'T$ with $K'$ and $H_\delta\cap K'T$ both maximal and core-free in $K'T$, or $(T,K'\cap T,T_\delta)$ is in~\cite[Table~1]{LPS4}. In the former case, we argue exactly as in the argument above with the group $H$ replaced by $K'T$ and we obtain $|\mathcal{T'}^{PA}|\le 2^b$. Therefore, we have to investigate the possibilities in~\cite[Table~1]{LPS4}. In all of the cases listed in \cite[Table 1]{LPS4}, the set $\Delta$ and the action of $T$ on $\Delta$ are explicitly described. Therefore, with another  case-by-case analysis and with routine computations we check \eqref{eqPAeqPA} and we see that, there exists a constant $a$ with $|G_S|\le a$. Moreover, one might take $a<b$ and hence $|\mathcal{T}'^{PA}|\le 2^b$ also in this case.
\end{proof}

\subsection{Estimating the cardinality of $\mathcal{T}'^{HS}\cup \mathcal{T}'^{HC}$}

\begin{theorem}\label{thrm:HSHC} We have $\mathcal{T}'^{HS}\cup\mathcal{T}'^{HC}=\emptyset$.
\end{theorem}
\begin{proof}
Let $S\in \mathcal{T}'^{HS}\cup\mathcal{T}'^{HC}$. 
 In both of these cases, $P_S=H\times K$ where $H$ and $K$ are isomorphic normal regular subgroups  of $G_S$. Since $(G_S)_1$ also acts regularly on $\Omega_S$, we deduce $|(G_S)_1|=|H|$. From the structure of primitive groups of HS and HC type~\cite{C3}, the stabiliser of a point of $\Omega_S$ in $G_S$ is isomorphic, as an abstract group, to a subgroup of $\Aut(H)$ containing the inner automorphisms of $H$. Therefore $|H|\le |R|\le |\Aut(H)|$. We deduce $$r=|R|\ge|H|= |(G_S)_1|\ge 2^{r^{0.499}}.$$ A simple calculation gives $|R|=r\le 16$. Thus $|H|\le 16$, but this is a contradiction because $H$ has size at least $|\Alt(5)|=60$.  Therefore $|\mathcal{T}'^{HS}\cup \mathcal{T}'^{HC}|=\emptyset$.
\end{proof}

\subsection{Estimating the cardinality of $\mathcal{T}'^{HA}\cup\mathcal{T}'^{SD}\cup\mathcal{T}'^{TW}$}

Before continuing our discussion on Cayley digraphs and using the theory of finite primitive groups for estimating the cardinality of $\mathcal{T}'$, we need an auxiliary result which is a refinement of a result of Liebeck and Praeger. We believe that this refinement is of considerable interest in its own. 

In~\cite{LPS2}, Liebeck and Praeger investigate the transitive subgroups of the finite primitive groups. This pioneer work highlights for the first time that, if $G$ is primitive and $M$ is a transitive subgroup of $G$ containing no non-identity normal subgroup of the socle of $G$, then $M$ is rather limited in its structure. We generalise, for regular subgroups only, the main result of Liebeck and Prager~\cite[Theorem~1]{LPS2}, when $G$ is of type SD or TW. We do believe that a similar generalisation holds for other classes of transitive subgroups, but we do not take this detour here. First we need some notation.

\smallskip

Suppose that $G$ is primitive on $\Omega$ of type SD. By the description of the O'Nan-Scott types in~\cite{C3}, there exists a non-abelian
simple group $T$ such that the socle $N$ of $G$ is isomorphic to $T_1 \times\cdots\times T_\ell$ with $T_i\cong T$ for each $i \in \{1,\ldots,\ell \}$. The set $\Omega$ can be identified with $T_1\times \cdots \times T_{\ell-1}$ and, for the
point $\omega\in \Omega$ that is identified with $(1, \ldots, 1)$, the stabilizer $N_\omega$ is the diagonal subgroup
$\{(t, \ldots, t) \mid t \in T \}$ of $N$. That is to say, the action of $N_\omega$ on $\Omega$ is permutation isomorphic
to the action of $T$ on $T^{\ell-1}$ by ``diagonal" component-wise conjugation: the image of the
point $(x_1 , \ldots, x_{\ell-1} )$ under the permutation corresponding to $t\in T$ is
$$(x_1^t,\ldots,x_{\ell-1}^t).$$
The group $G_\omega$ is isomorphic to a subgroup of $\Aut(T)\times \Sym(\ell)$ and $G$ is isomorphic to a subgroup of $T^\ell\cdot(\Out(T)\times \Sym(\ell))$.

\begin{proposition}\label{prop:SD}
Let $G$ and $N$ be as above and let $B$ be a regular subgroup of $G$. Then $B\le N\cdot \Out(T)$ and $B$ contains at least $\ell-3$ simple direct factors of $N$.
\end{proposition}
\begin{proof}
Using the notation that we have established above, $N\unlhd G\le W:=T^\ell\cdot (\Out(T)\times \Sym(\ell))$. Without loss of generality, for simplicity we may assume that $G=W$. 

We argue by induction on $\ell$ and we suppose first that $\ell=2$. In this case, $\ell-3=-1$ and hence the condition ``$B$ contains at least $\ell-3$ simple direct factors of $N$'' is satisfied vacuously. If $\ell=2$ and $B$ contains a simple direct factor of $N=T_1\times T_2$, then $T_1\le B$ or $T_2\le B$ and hence, since $B$ is regular, $B$ equals either $T_1$ or $T_2$. In particular, $B\le N\le N\cdot \Out(T)$. If $\ell=2$ and $B$ contains no simple direct factor of $N$, then $G$ and $B$ satisfy the hypothesis of~\cite[Theorem~1]{LPS2}. Now, we see that $B\le N\cdot\Out(T)$ from Remark~(2) on page 295 and Example~$1.2$ in~\cite{LPS2}.

Suppose now that $\ell>2$. Assume first that $B$ contains no simple direct factor of $N=T_1\times\cdots\times T_\ell$. Again, as above, $G$ and $B$ satisfy the hypothesis of~\cite[Theorem~1]{LPS2}. From~\cite[Theorem~1~(ii)]{LPS2}, we deduce $\ell=3$ and $B\le N\cdot \Out(T)$ from Remark~(2) on page 295 and Example~$1.3$ in~\cite{LPS2}. Therefore, we are done in this case.

Assume that $B$ contains some simple direct factor of $N$. Replacing $B$ by a suitable $G$-conjugate, we may assume that $T_1\le B$. Set $V:=\langle T_1^b\mid b\in  B\rangle$. Clearly, $V\cong T^\kappa$, for some $1\le\kappa\le \ell-1$. If $\kappa=\ell-1$, then $B=V$ because $B$ and $T^{\ell-1}$ act regularly on $\Omega$; thus $B\le N$ and $B$ contains $\ell-1$ simple direct factors of $N$. Assume then $\kappa\le\ell-2$. We have
$$\frac{B}{V}\le \frac{\norm W{V}}{V}\cong T^{\ell-\kappa}\cdot (\Out(T)\times \Sym(\ell-\kappa))$$
and the action of $\norm W{V}$ on the set of $V$-orbits on $\Omega$ is primitive with kernel $V$ and having O'Nan-Scott type SD. Moreover, in this action, $B/V$ is a regular subgroup of ${\norm W{V}}/{V}$. Thus our result follows immediately applying the inductive hypothesis to $B/V$ and ${\norm W{V}}/{V}$.
\end{proof}

A similar, but somehow weaker, proposition can be proved for finite primitive groups of TW type.
\begin{proposition}\label{prop:TW}
Let $G$ be a finite primitive group of TW type with socle $N=T^\ell$, where $T$ is a non-abelian simple group, and let $B$ be a regular subgroup of $G$. Then $|B:B\cap N|\le |\Aut(T)|$ and $B$ contains at least $\ell-3$ simple direct factors of $N$.
\end{proposition}
\begin{proof}
From the embeddings among the finite primitive groups, as $G$ is a primitive group of TW type, there exists a finite primitive group of SD type $W\cong T^{\ell+1}\cdot (\Out(T)\times \Sym(\ell))$ with $G\le W$. Applying Proposition~\ref{prop:SD} to $W$ and $B$ we deduce that $B$ contains at least $(\ell+1)-3=\ell-2$ simple direct factors of the socle of $W$ and hence $B$ contains at least $\ell-3$ simple direct factors of 
the socle of $G$. Moreover, $B\le T^{\ell+1}\cdot \Out(T)$ and hence $|B:B\cap N|=|BN:N|\le |T^{\ell+1}\cdot\Out(T):T^\ell|=|\Aut(T)|$.
\end{proof}

We do not believe that Proposition~\ref{prop:TW} is best possible. It is enough for our purpose and our proof actually follows immediately from the analogous result for primitive groups of SD type.

After this short detour of Propositions \ref{prop:SD} and \ref{prop:TW}, we go back to estimating $|\mathcal{T}'|$. 
\begin{proposition}\label{propositionnew}
For each $S\in \mathcal{T}'^{HA}\cup\mathcal{T}'^{SD}\cup\mathcal{T}'^{TW}$,  $|1^{P_S}|\le r^{0.501}(\log_2(r))^2$ where $1^{P_S}$ is the $P_S$-orbit containg $1$ in the action of $P_S$ on the vertices of $\Cay(R,S)$.
\end{proposition}
\begin{proof}
We consider a case-by-case analysis depending on the O'Nan-Scott type of $G_S$ in its action on $\Omega_S$.

\smallskip

\noindent\textsc{Case} $S\in\mathcal{T}'^{HA}$. 

\smallskip

\noindent Here, $P_S$ is an elementary abelian  $p$-group of cardinality $p^\ell$, for some prime number $p$ and for some positive integer $\ell$, acting regularly on $\Omega_S$. Moreover, $G_S=P_S\rtimes R$, with $R$ acting irreducibly by conjugation as a linear group on $P_S$. Since $(G_S)_1$ is also regular on $\Omega_S$ and $(G_S)_1$ is not normal in $ G_S$, we have $(G_S)_1\ne P_S$ and $(G_S)_1P_S>P_S$. As $G_S=P_S\rtimes R$, there exists a non-identity $p$-sugroup $Q$ of $R$ with $$P_S(G_S)_1=P_S\rtimes Q.$$ In particular, $p\le |Q|\le |R|=r$.

 Since $Q$  is a $p$-group, the group action of $Q$ on $P_S$ fixes a non-identity element $x\in P_S\setminus\{1\}$. Therefore $Q\le \cent {R}x$. Since $R$ acts irreducibly on $P_S$, the set $x^{R}=\{x^t\mid t\in R\}$ spans $P_S$ and so $\ell\le |x^R|=|R:\cent R x|\le |R:Q|$. 
 
 We have $p^\ell=|P_S|=|(G_S)_1|\ge 2^{r^{0.499}}$ and hence $\ell\log_2(p)\ge r^{0.499}$. Since $|R:Q|\ge \ell$ and $p\le r$, we deduce $$\frac{r}{|Q|}\log_2(r)\ge r^{0.499}$$ and $|Q|\le r^{0.501}\log_2(r)$.
 
  Since $|(G_S)_1|=|P_S|$, we have $|(G_S)_1P_S|=|(G_S)_1||P_S|/|(G_S)_1\cap P_S|=|P_S|^2/|(P_S)_1|$ from which it follows $$|P_S:(P_S)_1|=|(G_S)_1P_S:P_S|=|P_S\rtimes Q:P_S|=|Q|.$$ Thus $|1^{P_S}|=|P_S:(P_S)_1|=|Q|\le r^{0.501}\log_2(r)\le r^{0.501}(\log_2(r))^2$. 

\smallskip

\noindent\textsc{Case }$S\in \mathcal{T}'^{SD}$. 

\smallskip

\noindent We use the notation that we have established above for primitive groups of diagonal type. Thus $P_S=T^\ell$ for some non-abelian simple group $T$ and for some positive integer $\ell$ with $\ell\ge 2$. Since $(G_S)_1$ acts regularly on $\Omega_S$, we have $|(G_S)_1|=|T|^{\ell-1}$. From Proposition~\ref{prop:SD} applied to the regular subgroup $(G_S)_1$, we infer $(G_S)_1\le T^\ell\cdot \Out(T)$. Thus $(G_S)_1P_S\le T^\ell\cdot \Out(T)$ and hence $|(G_S)_1P_S|\le |T|^{\ell}|\Out(T)|$. Since $|(G_S)_1P_S:(G_S)_1|=|P_S:(G_S)_1\cap P_S|$, we deduce 
\begin{equation}\label{eq:help3}
|P_S:(P_S)_1|=\frac{|(G_S)_1P_S|}{|(G_S)_1|}\le \frac{|T|^\ell|\Out(T)|}{|T|^{\ell-1}}= |T||\Out(T)|.
\end{equation} Now, $|T|^{\ell-1}=|(G_S)_1|\ge 2^{r^{0.499}}$ and hence 
\begin{equation}\label{eq:help1}
\ell \log_2(|T|)>(\ell-1)\log_2(|T|)\ge r^{0.499}.
\end{equation} 

Recall that the stabiliser of a point in a primitive group of SD type contains the diagonal of $P_S$ and projects to a subgroup of $\Sym(\ell)$ acting transitively on the $\ell$ simple direct summands of the socle $P_S=T^\ell$. Thus, we obtain the bound
\begin{equation}\label{eq:help2}
r=|R|\ge |T|\ell.
\end{equation} 

From~\eqref{eq:help3},~\eqref{eq:help1} and~\eqref{eq:help2}, we deduce
\begin{align*}
|1^{P_S}|&=|P_S:(P_S)_1|\le |T|\log_2(|T|)\le\frac{r}{\ell}\log_2(|T|)\le 
\frac{r}{ \frac{r^{0.499}}{\log_2(|T|)} }\log_2(|T|)\\
&=r^{0.501}(\log_2(|T|))^2\le r^{0.501}(\log_2(r))^2.
\end{align*}
 
\smallskip

\noindent\textsc{Case }$S\in \mathcal{T}'^{TW}$.

 \smallskip
 
  \noindent We use the notation that we have established above for primitive groups of TW type. Thus $P_S=T^\ell$ for some non-abelian simple group $T$ and for some positive integer $\ell$ with $\ell\ge 6$, see~\cite{C3}. Since $P_S$ and $(G_S)_1$ act regularly on $\Omega_S$, we have $|P_S|=|(G_S)_1|=|T|^{\ell}$ and hence 
\begin{align}\label{eq:006}
|1^{P_S}|=|P_S:(P_S)_1|=|G_S:(G_S)_1\cap P_S|.
\end{align} From Proposition~\ref{prop:TW} applied to the regular subgroup $(G_S)_1$, we infer $(G_S)_1\le T^\ell\cdot \Out(T)$ and hence
\begin{align}\label{eq:007}
|(G_S)_1:(G_S)_1\cap P_S|\le |\Aut(T)|.
\end{align} Now, $|T|^{\ell}=|(G_S)_1|\ge 2^{r^{0.499}}$ and hence $1\le \ell\log_2(|T|)/r^{0.499}$, that is, 
\begin{align}\label{eq:008}|\Aut(T)|\le |\Aut(T)|
\frac{
\ell\log_2(|T|)}{r^{0.499}}.
\end{align} Recall that the stabiliser of a point in a primitive group of TW type acts transitively on the $\ell$ simple direct summands of the socle $P_S=T^\ell$ and contains a subgroup isomorphic to $T$ normalizing one of the simple direct summands of $P_S$, see~\cite{LPSLPS}. Thus, the inequality in~\eqref{eq:help2} holds true also in this case. Hence, from~\eqref{eq:006},~\eqref{eq:007} and~\eqref{eq:008}, we obtain
\begin{align*}
|1^{P_S}|&\le\frac{|\Aut(T)|\ell\log_2(|T|)}{r^{0.499}}\le\frac{|T|\ell(\log_2(|T|))^2}{(\ell|T|)^{0.499}}=(\ell|T|)^{0.501}(\log_2(|T|))^2\le r^{0.501}(\log_2(r))^2.
\end{align*}
Observe that in the second inequality we have used the crude upper bound $|\Out(T)|\le\log_2(|T|)$, which follows easily from the CFSG.
\end{proof}

For our next result we need the notion of \textit{normal quotient} for digraphs.
\begin{definition}\label{normalquotient}{\rm
Let $\Gamma$ be a digraph, let $G$ be a group of automorphisms of $\Gamma$ transitive on the vertices of $\Gamma$ and let $N$ be a normal subgroup of $G$. Let $\alpha^N$ denote the $N$-orbit containing the vertex $\alpha$ of $\Gamma$. Then the
\textit{normal quotient} $\Gamma/N$ is the digraph whose vertices are the $N$-orbits on the vertices of $\Gamma$, with a directed
edge from $\alpha^N$ to $\beta^N$ if and only if there is a directed edge of $\Gamma$ from
$\alpha'$ to $\beta'$, for some $\alpha'\in\alpha^N$ and some $\beta'\in\beta^N$. The normal quotient is non-trivial
if $N \ne 1$ and $N$ is not transitive.

Note that the group $G$ acts as a group of automorphisms on $\Gamma/N$ and induces a transitive action on the vertices of the normal
quotient $\Gamma/N$. Also, for adjacent $\alpha^N$, $\beta^N$ of $\Gamma/N$, each vertex of $\alpha^N$ is adjacent to the same number of vertices in $\beta^N$ (because $N$ is transitive on both sets). Moreover, the stabiliser in $G$ of the vertex $\alpha^N$  in $\Gamma/N$ is $G_\alpha N$.
}
\end{definition}

\begin{theorem}\label{thrmpartial}
We have $|
\mathcal{T}'^{HA}\cup\mathcal{T}'^{SD}\cup\mathcal{T}'^{TW}|\le 2^{r-\frac{r^{0.499}}{8(\log_2(r))^2}+2(\log_2(r))^2+1}$.
\end{theorem}
\begin{proof}
For simplicity, write $\mathcal{T}'':=
\mathcal{T}'^{HA}\cup\mathcal{T}'^{SD}\cup\mathcal{T}'^{TW}$ and we let $S\in\mathcal{T}''$. We have $P_S\unlhd G_S\le\Aut(\Cay(R,S))$ and hence $\Cay(R,S)$ admits a normal quotient by the normal subgroup $P_S$ of $G_S$; let us denote by $\Cay(R,S)/P_S$ this normal quotient. Since $R$ acts regularly on the vertices of $\Cay(R,S)$, the system of imprimitivity given by the $P_S$-orbits on the vertices of $\Cay(R,S)$  coincides with  the set of cosets of $R$ via a suitable subgroup $Q_S$ of $R$ with $$(G_S)_1P_S=(G_S)_1Q_S,$$ indeed $Q_S:=1^{P_S}$, that is, the $P_S$-orbit containing the identity element of $R$. Now, the group $R$ acts as a group of automorphisms on the graph $\Cay(R,S)/P_S$ with vertex stabilizer $R_1Q_S=Q_S$.

Now define 
\begin{align*}
\mathcal{T''}_{\unlhd}&:=\{S\in\mathcal{T}''\mid Q_S\unlhd R\},&\mathcal{T''}_{\not\unlhd}:=\{S\in \mathcal{T''}\mid Q_S\not\!\!\unlhd R\}.
\end{align*}
From Proposition \ref{propositionnew}, for each $S\in \mathcal{T''}$, we have $$|Q_S|=|1^{P_S}|\le r^{0.501}(\log_2(r))^2. $$

\smallskip

\noindent\textsc{Claim 1:  }$|\mathcal{T''}_{\not\unlhd}|\le 2^{r-\frac{r^{0.499}}{4(\log_2(r))^2}+(\log_2(r))^2}$.

\smallskip 



\noindent We start our argument by estimating, given a  non-normal subgroup $Q$ of $R$, the number of subsets $S$ of $R$ such that the cosets of the subgroup $Q$  partition  the vertices of $\Cay(R,S)$ into $Q$-cosets  forming a \textit{normal} system of imprimitivity, that is, a system of imprimitivity  arising also  from the orbits of a normal subgroup of $\Aut(\Cay(R,S))$. Then, the proof of this claim immediately follows because $R$ has at most $2^{(\log_2(r))^2}$ subgroups.

Let $S$ be a subset of $R$ and let $Q$ be a non-normal subgroup of $R$ with the property that the $Q$-cosets form a normal system of imprimitivity for the graph $\Cay(R,S)$.   Let $\Delta:=Q\backslash R$ and $S_\delta:=S\cap \delta$, for each $\delta\in\Delta$. Fix $q\in Q$. Since the outneighbourhood of $1$ in $\delta$ is $S\cap \delta=S_\delta$, the outneighbourhood of $q$ in $\delta q$ is $(S\cap \delta)q=S_q\cap \delta q$. However, since $ \Cay(R,S)/P_S$ is a normal quotient graph (see Definition \ref{normalquotient}), $1$ and $q$ have the same number of outneighbours in $\delta q$ and hence
$$|S\cap \delta q|=|Sq\cap \delta q|.$$
Clearly this yields 
$$|S\cap \delta q|=|S\cap\delta|, \textrm{ for each }\delta\in \Delta \textrm{ and for each }q\in Q.$$ In other words, the cardinality of $S\cap \delta$ depends only on the orbits of $Q$ on $\Delta$, that is, 
\begin{center}if $\delta_1,\delta_2\in \Delta$ are in the same $Q$-orbit, then $|S_{\delta_1}|=|S_{\delta_2}|$.
\end{center} Let us denote by $d_1,\ldots,d_o$ the cardinality of the orbits of $Q$ on $\Delta$;  observe that $d_1,\ldots,d_o$  are not all equal to $1$ because $Q$ is not acting trivially on $\Delta$  for $Q\not\!\!\unlhd R$. Clearly, $|\Delta|=|R:Q|=\sum_{i=1}^od_i$. 

Let $\delta_1,\ldots,\delta_o$ be a set of representatives for the orbits of $Q$ on $\Delta$. From the previous paragraph, for each $q\in Q$, $S_{\delta_i}$ and $S_{\delta_i q}$ have the same cardinality and hence, in particular, they have the same parity modulo $2$. Therefore, to obtain an upper bound on the number of subsets $S$ of $R$ with $\Cay(R,S)$ admitting a normal quotient arising from the $Q$-cosets, we may choose an arbitrary subset $S_i$ of $\delta_i$, for each $i\in \{1,\ldots,o\}$, and for each $\delta \in \delta_i^Q$, we may choose a subset of $\delta$ having the same parity of $S_i$. In this manner, we obtain the following over-estimate on the number of possibilities for $S$:

\begin{align*}
2^{o|Q|+\sum_{i=1}^o(|Q|-1)(d_i-1)}&=2^{o|Q|+(|Q|-1)(|R:Q|-o)}=2^{|R|-|R:Q|+o}\\
&\le 2^{|R|-|R:Q|+\frac{3}{4}|R:Q|}=2^{|R|-\frac{1}{4}|R:Q|}\le 2^{r-\frac{r^{0.499}}{4(\log_2(r))^2}}.
\end{align*}
In the first inequality above, we are using Lemma~\ref{lemma1}.~$_\blacksquare$

\smallskip

\noindent\textsc{Claim 2: }$|\mathcal{T''}_{\unlhd }|\le 2^{r-\frac{r^{0.499}}{8(\log_2(r))^2}+2(\log_2(r))^2}$. 

\smallskip

\noindent Let $S\in\mathcal{T''}_{\unlhd }$. Clearly, $P_S\nleq R$, otherwise $G_S=RP_S=R$. Let $f\in P_S\setminus R$ and observe that $f$ fixes each $P_S$-orbit setwise and hence it fixes each $Q_S$-coset setwise. Since $Q_S\unlhd R$, we are in the position to apply Theorem \ref{red1red1} with the normal subgroup $Q_S$ of $R$ and with the automorphism $f$. We obtain that 
\begin{align*}
|\mathcal{T''}_{\unlhd}|&\le 2^{r-\frac{r/n-2}{3}\log_2(4/3)+(\log_2(r))^2+\log_2(r)+\log_2(n)-1}\\
&\le 2^{r-\frac{r^{0.499}}{3(\log_2(r))^2}\log_2(4/3)+(\log_2(r))^2+\log_2(r)+(\log_2(r^{0.501}(\log_2(r))^2))}\\
&\le 2^{r-\frac{r^{0.499}}{8(\log_2(r))^2}+2(\log_2(r))^2},
\end{align*}
where the second inequality follows because $\frac{2}{3}\log_2(4/3)-1<0$, and the last inequality follows from the facts that $\frac{1}{3}\log_2(4/3) >1/8. ~_\blacksquare$
\smallskip

The proof now follows by adding the bounds given in the previous two claims.
\end{proof}


\subsection{Estimating the cardinality of $\mathcal{T}'^{CD}$} We start by reviewing the structure of primitive groups of CD type. 
Let $S\in\mathcal{T}'^{CD}$.
Here, $G_S$ is contained in a wreath product $H\mathrm{wr}\Sym(\kappa)$ endowed of its natural product action on $\Delta^\kappa$, where $\kappa\ge 2$ and $H$  is a primitive group of SD type on $\Delta$. Thus, using the notation for primitive groups of SD type that we established above, there exists a positive integer $\ell$ and a non-abelian simple group $T$ with $$T^\ell \le H\le T^\ell\cdot (\Out(T)\times \Sym(\ell)).$$ Now, we denote by $Q$ the projection of $H$ to $\Sym(\ell)$. The group $Q$ can be described more formally. The socle $P_S\cong T^{\kappa\ell}$ of $G_S$ is
$$(T_{1,1}\times \cdots \times T_{1,\ell})\times (T_{2,1}\times \cdots\times T_{2,\ell})\times \cdots \times (T_{\kappa,1}\times \cdots\times T_{\kappa,\ell}).$$
Take $V:=T_{1,1}\times \cdots \times T_{1,\ell}$ and $W:=T_{1,1}$. Now, from the structure of primitive groups of CD type, $\norm {G_S} V$ has index $\kappa$ in $G_S$. Moreover, replacing $H$ by a suitable subgroup, we may assume that $$\norm {G_S} V\le H\times (H\mathrm{wr}\Sym(\kappa-1))$$ projects surjectively to $H$. Similarly, $\norm {G_S} W$ has index $\kappa \ell$ in $G_S$ and  $\norm {G_S} V/\norm {G_S} W$ projects to a primitive subgroup of $\Sym(\ell)$, which we denote by $Q$. Clearly, all the subgroups of $G_S$ we have defined so far (for instance, $H$, $V$, $W$ and $Q$) depend on $S$ because so does $G_S$, but to avoid making the notation too cumbersome to use, we do not stress this.

 Recall that $(G_S)_1$ is transitive on $\Omega_S$ and hence $$|(G_S)_1|=|\Omega_S|=|T|^{\kappa(\ell-1)}.$$ Moreover, as $G_S=RP_S$ and $P_S\le \norm {G_S}W\le \norm {G_S}V$, we deduce that $|R:\norm R V|=\kappa$ and that $\norm R V/\norm R W$ projects to $Q$.  Since $\norm R W\ge R\cap P_S\cong T^\kappa$, we get $$|R|=|R:\norm R{V}||\norm R V:\norm R W||\norm R{W}|\ge \kappa|Q||R\cap P_S|\ge \kappa|Q||T|^\kappa$$ and hence we obtain the inequality
$$|T|^{\kappa(\ell-1)}=|(G_S)_1|\ge 2^{|R|^{0.499}}\ge 2^{(\kappa|Q||T|^\kappa)^{0.499}}.$$
 Rearranging the terms, we get
\begin{equation}\label{eq:CD1}
\frac{\ell-1}{|Q|^{0.499}}
\ge 
\frac{|T|^{0.499\cdot\kappa}\kappa^{-0.501}}{\log_2(|T|)}.
\end{equation}
Moreover, since $Q$ is a transitive subgroup of $\Sym(\ell)$, we have $|Q|\ge \ell$ and hence, from  the inequality $|R|\ge \kappa|Q||T|^\kappa\ge \kappa\ell|T|^\kappa$, we obtain
\begin{equation}\label{eq:CD2}
\kappa\le \frac{\log_2(r/4)}{\log_2(60)},\quad \ell\le \frac{r}{7200},\quad |T|\le \frac{\sqrt{r}}{2}.
\end{equation}
The inequalities in~\eqref{eq:CD2} are all easy to obtain: for instance, since $r\ge \kappa\ell|T|^\kappa$, $\ell,\kappa\ge 2$ and $|T|\ge 60$, we have $r\ge 4 \cdot 60^\kappa$ and hence $\kappa\le \log_2(r/4)/\log_2(60)$.

We claim that
\begin{equation}\label{eq:CD3}
|Q|< \ell^{2.01}.
\end{equation}
Suppose, arguing by contradiction, that $|Q|\ge \ell^{2.01}$. From~\eqref{eq:CD1}, we deduce 
$$1\ge \frac{\ell-1}{\ell^{1.00299}}
\ge\frac{\ell-1}{|Q|^{0.499}}\ge 
\frac{|T|^{0.499\cdot\kappa}\kappa^{-0.501}}{\log_2(|T|)};$$
however, with a simple calculation we see that this inequality is never satisfied.

We are now ready to conclude our analysis on the cardinality of $\mathcal{T'}^{CD}$.
\begin{theorem}\label{CD1}
There exists an absolute constant $c$ such that
$$|\mathcal{T}'^{CD}|\le 2^{\frac{3}{4}r+
(\log_2(r))^4+(\log_2(r))^3/5+(c+1)(\log_2(r))^2+\log_2(r)}.$$
Moreover, one might take the costant $c$ to be equal to the constant in Theorem I in \cite{PS}.
\end{theorem}
\begin{proof}
We use the notation we have established above and, in particular, \eqref{eq:CD2} and \eqref{eq:CD3}. For each $S\in\mathcal{T}'^{CD}$,  $G_S$ is of CD type and hence there exist
\begin{itemize}
\item a non-abelian simple group $T$ with $|T|\le \sqrt{r}$/2, and
\item  some  positive integers $\ell,\kappa\ge 2$  with $\ell\le r/7200$ and $\kappa\le \log_2(r/4)/\log_2(60), $ and
\item a primitive subgroup $Q$ of $\Sym(\ell)$ with $|Q|< \ell^{2.01}$,
\end{itemize}
such that
$$G_S\le W:=(T^\ell\cdot (\Out(T)\times Q))\mathrm{wr}\Sym(\kappa)$$
endowed of its natural compound diagonal action on $\Omega_S$. 
%

A fundamental result of Pyber and Shalev \cite[Theorem I]{PS} shows that there exists an absolute constant $c$ such that the number of conjugacy classes of primitive subgroups of $\Sym(\ell)$ is at most
$2^{c(\log_2(\ell))^2}$. We deduce the following:

\smallskip

\noindent\textsc{Fact 1: }The number of possibilities for the group $W$, up to isomorphism, is at most
$$2 |T|\ell\kappa2^{c(\log_2(r))^2}\le \kappa\ell |T|^\kappa\cdot 2^{c(\log_2(r))^2}\le 2^{c(\log_2(r))^2+\log_2(r)}.$$
Observe that the factor $2$ in front of $\sqrt{r}/2$ accounts for the fact that for each natural number $x$, there exist at most two non-abelian simple groups having order $x$. In the last inequality we are using the inequality $r \ge \kappa \ell |T|^\kappa$.

\smallskip

Observe that $W/P_S\cong (\Out(T)\times Q)\mathrm{wr}\Sym(\kappa)$. Checking the order of the outer-automorphism group of a non-abelian simple group, we have $|\Out(T)|\le \log_2(|T|)$ and, since $\kappa!\le \kappa^\kappa$, we  obtain
\begin{align*}
|W:P_S|&=(|\Out(T)||Q|)^\kappa \kappa!\le (\log_2(|T|)\ell^{2.01})^\kappa\kappa!\le(\log_2(|T|)\ell^{2.01}\kappa)^\kappa\\
&\le 2^{\frac{\log_2(r/4)}{\log_2(60)}[\log_2(\log_2(\sqrt{r}/2))+2.01\cdot\log_2(r/7200)+\log_2\left(\log_2(r/4)/\log_2(60)\right)]}\\
&\le 2^{(\log_2(r))^2/5}.
\end{align*}
Observe that $P_S$ is the socle of $W$ and hence it is uniquely determined by $W$. Now, $P_S\le G_S=P_SR\le W$ and the group $R$ has at most $\lfloor \log_2(r)\rfloor$ generators, therefore the number of choices for 
$G_S$ is at most $|W/P_S|^{\log_2(r)}=2^{(\log_2(r))^3/5}$. 
Combining this with Fact 1, we obtain

\smallskip

\noindent\textsc{Fact 2: }The number of possibilities for the abstract group $G_S$, up to isomorphism, is at most $2^{(\log_2(r))^3/5+c(\log_2(r))^2+\log_2(r)}$.

\smallskip

Observe that $P_S$ is the socle of $G_S$ and hence it is uniquely determined by $G_S$.  Let $C$ be the core of $(G_S)_1\cap P_S=(P_S)_1$ in $P_S$, see Figure \ref{pic:1}.
As $C\unlhd P_S=T^{\kappa\ell}$, we have $P_S/C\cong T^s$, for some positive integer $s$. Therefore, we have 
$$
{\kappa \ell \choose s}\le (\kappa \ell)^s$$
choices for $C$.
As $|G_S:(G_S)_1|=r$, we deduce $|P_S:(P_S)_1|\le r$ and hence $P_S/C\cong T^s$ has a faithful permutation representation on the set of right cosets of $(P_S)_1$ in $P_S$ of degree at most $r$. From \cite[Theorem 3.1]{EP}, the minimal degree of a faithful permutation representation of $T^s$  is $(m(T))^s$, where $m(T)$ is the minimal degree of a faithful permutation representation of the simple group $T$. Clearly, $m(T)\ge 5$. Therefore, we have $$5^s\le m(T)^s\le r$$ and hence $5^s\le r$. From this we deduce $$s\le \log_2(r)/\log_2(5).$$
Therefore, the number of choices for $C$ is at most
$$\frac{\log_2(r)}{\log_2(5)}\cdot (\kappa \ell)^{\frac{\log_2(r)}{\log_2(5)}}
\le \frac{\log_2(r)}{\log_2(5)}\cdot 2^{\frac{\log_2(r)}{\log_2(5)}[\log_2(\log_2(r/4)/\log_2(60))+\log_2(r/7200)]}
\le 2^{(\log_2(r))^2}.$$
Moreover, $$|G_S:C|=|G_S:(G_S)_1||(G_S)_1:C|=r|T|^{s-\kappa}\le r(\sqrt{r}/2)^{\frac{\log_2(r)}{\log_2(5)}}\le 2^{(\log_2(r))^2}.$$
\begin{figure}[!h]
\begin{tikzpicture}[node distance   = 1cm ]
\tikzset{
    myarrow/.style={==, thick}
}  
        \node(A){$G_S$};
\node[below=of A](B){$(G_S)_1P_S$};
\node[below=of B](C){$(G_S)_1$};
\node[right=of C](D){$P_S$};
\node[below=of D](E){$(G_S)_1\cap P_S=(P_S)_1$};
\node[below=of E](F){$C$};
\draw(A)--(B);
\draw(B)--(C);
\draw[myarrow](B)--(D);
\draw[myarrow](E)--(C);
\draw(E)--(D);
\draw(E)--(F);
\end{tikzpicture}
\caption{Subgroup lattice for $G_S$}\label{pic:1}
\end{figure}
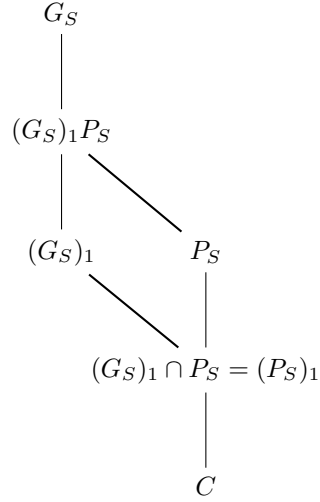
The group $(G_S)_1$ is contained between $G_S$ and $C$ and hence we have at most $|G_S:C|^{\log_2(|G_S:C|)}=2^{(\log_2(|G_S:C|))^2}\le 2^{(\log_2(r))^4}$ choices for the subgroup $(G_S)_1$, when the subgroup $C$ is given. Summing up, we have proved the following:

\smallskip

\noindent\textsc{Fact 3: }Given the group $G_S$ as an abstract group, the number of choices for $(G_S)_1$ is at most
$$2^{(\log_2(r))^4+(\log_2(r))^2}.$$

\smallskip

Combining Facts 2 and 3, we have that 
$$|
\{
G_S\mid S\in \mathcal{T}'^{CD}
\}
|
\le 
2^{
(\log_2(r))^4+(\log_2(r))^3/5+(c+1)(\log_2(r))^2+\log_2(r)
}.$$
Now, the proof follows immediately from Lemma \ref{lemma1}: each permutation group in $\{G_S\mid S\in\mathcal{T'}^{CD}\}$ is acting as a group of automorphisms  on at most $2^{3r/4}$ graphs.
\end{proof}

\subsection{Pulling the threads together}
Summing up, in this section we have proved the following result.
\begin{theorem}\label{primitive} There exist two positive constants $a'$ and $b'$ with $|\mathcal{T}'|\le 2^{r-a'r^{0.499}/(\log_2(r))^2}+b'$, whenever $r \ge r_\varepsilon$.
\end{theorem}

\begin{corollary}\label{cor:prim} There exist two positive constants $b$ and $r_\varepsilon'$ with $|\mathcal T'| \le 2^{r-br^{0.499}/(\log_2(r))^2}$, whenever $r \ge r_\varepsilon'$.
\end{corollary}

\begin{proof}
This follows from Theorem~\ref{primitive} by choosing a value for $b$ that is smaller than $a'$, with the result that the resulting increase to the power of $2$ compensates for not adding the constant $b'$.
\end{proof}

\section{The remaining sets in $\mathcal{T}$}\label{sec:imprimitive}

In the previous section, we dealt with the sets in $\mathcal{T}'$; that is, the subsets $S\subseteq R$ that satisfy~\ref{hyp1} and admit a subgroup $G\le\Aut(\Cay(R,S))$ satisfying~\ref{hyp2},~\ref{hyp3},~\ref{hyp4} and~\ref{hyp5}. We showed that the number of such sets is negligible compared to $2^{|R|}$. It remains to be shown that the total number of sets in $\mathcal T$, that is, those that satisfy~\ref{hyp1} and admit a subgroup $G\le\Aut(\Cay(R,S))$ satisfying~\ref{hyp2},~\ref{hyp3} and~\ref{hyp4} but not necessarily~\ref{hyp5}, is also negligible. This is the goal of this section.

We begin by observing that since we have already counted the sets in $\mathcal T'$, we need only count the subsets in $\mathcal T \setminus \mathcal T'$; that is, subsets $S$ that satisfy~\ref{hyp1} and admit a subgroup $G\le\Aut(\Cay(R,S))$ satisfying~\ref{hyp2},~\ref{hyp3} and~\ref{hyp4} with the additional property that the core $G_R$ (of $R$ in $G$) is non-trivial. (Note that some of these subsets $S$ might also satisfy~\ref{hyp5} with a different choice of the subgroup $G\le \Aut(\Cay(G,S))$, but this only means that we might be counting some sets twice; the upper bound we arrive at will still be valid.)

In this section, we will rely on applying the results achieved in Section~\ref{sec:stopit} to particular quotients of the group $R$. Therefore, to avoid possible misunderstandings,  we use the notation $\mathcal{T}(R)$ and $\mathcal{T}'(R)$ for emphasising the ambient group $R$.


As mentioned above, if $S \in \mathcal T(R) \setminus \mathcal T'(R)$, then there exists a subgroup $G \le \Aut(\Gamma(R,S))$ with
\begin{description}
\myitem{(H2)} $R$ maximal in $G$ and $|G_1|\ge 2^{|R|^{0.499}}$,
\myitem{(H3)} $|G_R|\le 4\log_2(|R|)$,
\myitem{(H4)} some $G_R$-orbit is not fixed (setwise) by $G_1$,
\myitem{($\lnot$H5)}\label{hypnot5} the core $G_R$ of $R$ in $G$ is non-trivial.
\end{description}  Fix any such $G$. By~\ref{hyp4} with this $G$, we see that $G_R \neq R$ (since $RG_1=G \trianglelefteq G$), so we have $1 <G_R<R$. Thus, the orbits of $G_R$ form a non-trivial system of imprimitivity for $G$. There is a traditional definition of a quotient graph that can be formed in such a case; indeed, we have introduced this normal quotient in Definition~\ref{normalquotient} and we already used some of its properties in Theorem~\ref{thrmpartial}. However, to make our argument work, we define a different quotient graph.

\begin{definition}{\rm 
Let $\Gamma$ be a digraph whose vertex set $V$ has been partitioned into a collection of sets, $\mathcal B$, with the additional property that given any two sets $B, B' \in \mathcal B$, and any vertex $v \in B$, the number of arcs from $v$ to $B'$ does not depend on the choice of $v \in B$. Define the \emph{odd quotient digraph} of $\Gamma$  with respect to the partition $\mathcal B$ to be the digraph whose vertices are the sets $B \in \mathcal B$, with an arc from $B$ to $B'$ if and only if the number of arcs from each $v \in B$ to $B'$ is odd. }
\end{definition}

Clearly since $G_R$ acts transitively on its orbits while fixing each of them setwise, the number of arcs from any vertex in one orbit to any other orbit does not depend on the choice of the vertex, so we can form the odd quotient digraph of $\Gamma:=\Cay(R,S)$ with respect to the orbits of $G_R$. We denote this odd quotient by $\Gamma^{\text{odd}}_{G_R}$. Notice that any automorphism of $\Gamma$ induces an automorphism of $\Gamma^{\text{odd}}_{G_R}$.

As $R/G_R$ acts regularly on the vertices of $\Gamma^{\text{odd}}_{G_R}$, we observe that $\Gamma^{\textrm{odd}}_{G_R}$ is a Cayley digraph on $R/G_R$, say $\Gamma^{\textrm{odd}}_{G_R}=\Cay(R/G_R,S')$. Moreover, $G/G_R$ acts as a group of automorphisms of $\Gamma^{\textrm{odd}}_{G_R}$. Let $K$ be the kernel of the action of $G/G_R$ on the vertices of $\Gamma^{\textrm{odd}}_{G_R}$. Then $$K=\bigcap_{g\in G}(G_1G_R)^g.$$
Now, $K$ is a normal subgroup of $G$ fixing each $G_R$-orbit setwise. Therefore, by~\ref{hyp4}, $K_1=1$, that is, $K=G_R$ and $G/G_R$ acts faithfully on the vertices of $\Gamma^{\textrm{odd}}_{G_R}$. So, in what follows, we may regard $G/G_R$ as a subgroup of $\Aut(\Gamma^{\text{odd}}_N)$. 

Since $R$ is maximal in $G$, $R/G_R$ is maximal in $G/G_R$. Moreover, $|G_1G_R:G_R|=|G_1|\ge 2^{|R|^{0.499}}\ge 2^{|R/G_R|^{0.499}}$. Therefore, $$G/G_R \textrm{ satisfies~\ref{hyp2}}.$$

Since $G_R$ is the core of $R$ in $G$, we obtain that $G_R/G_R=1$ is the core of $R/G_R$ in $G/G_R$, that is, $R/G_R$ is core-free in $G/G_R$ and hence
$$G/G_R \textrm{ satisfies~\ref{hyp3},~\ref{hyp4} and~\ref{hyp5}}.$$

In particular, $S'\in \mathcal{T}'(R/G_R)$ (recall that $S'$ is the connection set for the odd quotient graph) and we are in a position to apply the main results of Section~\ref{sec:stopit} to the quotient group $R/G_R$.
\begin{theorem}\label{imprimitive}
With the choice of $\varepsilon$ from the start of Section~$\ref{sec44}$, there is a value $r''_\varepsilon$ and a positive constant $b$  such that for every $r \ge r''_\varepsilon$ and for every regular subgroup $R$ of 
$\Sym(r)$,
the number of subsets $S$ in $\mathcal T(R) \setminus \mathcal T'(R)$ is at most $2^{r-br^{0.499}/(4(\log_2(r))^3)+1}$.
\end{theorem}

\begin{proof}
We use the notation laid out in this section, and for any $S \in \mathcal T(R) \setminus \mathcal T'(R)$ form the Cayley graph $\Gamma:=\Gamma(R,S)$, and choose some fixed $G \le \Aut(\Gamma)$ that satisfies~\ref{hyp2},~\ref{hyp3},~\ref{hyp4} and~\ref{hypnot5}. Set $n:=|G_R|$, and form the odd quotient graph $\Gamma^{\text{odd}}_{G_R}$. Since $S \notin \mathcal T'(R)$, $G_R$ is non-trivial. Define $S'$ to be the connection set for $\Gamma^{\text{odd}}_{G_R}$ viewed as a Cayley digraph over $R/G_{R}$, so $\Gamma^{\text{odd}}_{G_R}=\Cay(R/G_R,S')$.

From the discussion preceding the statement of this theorem, $S' \in \mathcal T'(R/G_R)$.
Therefore, by Corollary~\ref{cor:prim}, when $r/n\ge r_\varepsilon'$, the number of choices for $S'$ is at most \begin{eqnarray*}&&2^{r/n-b(r/n)^{0.499}/(\log_2(r/n))^2},\end{eqnarray*} for some positive constant $b$.

The cardinality of $\mathcal T(R) \setminus \mathcal T'(R)$ (which we are trying to count) is the number of choices for $S$. By this reduction, this value is the number of choices for $S'$, times the product over all distinct cosets $gG_R$ of $G_R$ in $R$, of the number of choices for $S\cap gG_R$ that lead to $gG_R$ being in or not in $S'$, as appropriate. We have bounded the number of choices for $S'$; now we consider the number of choices for $S\cap gG_R$ that lead to $gG_R$ being in or not in $S'$, as appropriate.

By our construction of $\Gamma^{\text{odd}}_{G_R}$, any connection set $S'$ for $\Gamma^{\text{odd}}_{G_R}=\Cay(R/G_R,S')$ comes from any connection set $S$ for $\Gamma$ that satisfies the following conditions: $S\cap G_R$ can be any subset of $G_R$; and for any $g \in R\setminus G_R$, $S \cap gG_R$ must have odd cardinality if $gG_R \in S'$, and must have even cardinality if $gG_R \notin S'$. 

Recall that given a finite set $X$, the number of subsets of $X$ whose cardinality is even is equal to the number of subsets of $X$ whose cardinality is odd, and both are equal to $2^{|X|-1}$. 

Thus, the product over all distinct cosets $gG_R$ of $G_R$ in $R$, of the number of choices for $S\cap gG_R$ that lead to $gG_R$ being in or not in $S'$, as appropriate, is simply $$2^n(2^{n-1})^{r/n-1}=2^{r-r/n+1}.$$

We therefore conclude that, when $r/n\ge r_\varepsilon'$, the cardinality of $\mathcal T(R) \setminus \mathcal T'(R)$ is at most
$$2^{r-r/n+1}2^{
r/n-b(r/n)^{.499}/(\log_2(r/n))^2}=2^{r-b(r/n)^{.499}/(\log_2(r/n))^2+1}.$$ Since $n\ge 2$, this is no bigger than $$2^{r-b(r/n)^{.499}/(\log_2(r))^2+1}=2^{r-br^{.499}/(n^{.499}(\log_2(r))^2)+1}\le 2^{r-b'r^{.499}/(n(\log_2(r))^2)+1}.$$
Since $n \le 4\log_2(r)$, this is bounded above by 
$$2^{r-br^{.499}/(4(\log_2(r))^3)+1},$$ as claimed. 

To ensure that $r/n>r'_\varepsilon$, since $n \le 4 \log_2(r)$ it is sufficient to require $r/\log_2(r)>4r'_\varepsilon$. Since $r/\log_2(r)$ is an increasing function when $r>2$, we take $r''_\varepsilon$ large enough that $r''_\varepsilon/\log_2(r''_\varepsilon)=4r'_\varepsilon$.
\end{proof}

\begin{proof}[Proof of Theorems~$\ref{th:main1}$ and~$\ref{th:main2}$]
The proof follows immediately from Corollary~\ref{cor:prim} and Theorem~\ref{imprimitive}, observing that these bounds should be added and the bound from Corollary~\ref{cor:prim} is the smaller of the two, so that doubling the bound from Theorem~\ref{imprimitive} gives an overall bound.
\end{proof}

\section{Unlabeled digraphs}\label{sec:unlabelled}
An  \emph{unlabeled} (di)graph is simply an equivalence class of (di)graphs under the relation ``being isomorphic to''. We will often identify a representative with its class. Using this terminology, we have the following unlabeled version of Theorem~\ref{th:main1}.

\smallskip

\noindent\textbf{Theorem~\ref{th:unlabelledmain1}. }\textit{Let $R$ be a group of order $r$. Then the ratio of the number of unlabeled $\mathrm{DRR}$s on $R$ over the number of unlabeled Cayley digraphs on $R$ tends to $1$ as $r\to\infty$.}

\begin{proof}
For this proof, we let $\mathrm{CD}(R)$ denote the set of unlabeled Cayley digraphs on $R$, we let $\mathrm{DRR}(R)$ denote the set of unlabeled DRRs on $R$, we let $\mathrm{NDG}(R)$ denote the set of unlabelled Cayley digraphs on $R$ which are not DRRs, we let $2_{\mathrm{DRR}}^{R}$ denote the collection of the subsets $S$ of $R$ with $\Cay(R,S)$ a DRR and we let $2_{\mathrm{NDG}}^{R}$ denote the collection of the subsets $S$ of $R$ with $\Cay(R,S)$ not a DRR. In particular, $\mathrm{CD}(R)=\mathrm{DRR}(R)\cup\mathrm{NDG}(R)$ and $2^R=2^{R}_{\mathrm{DRR}}\cup 2^{R}_{\mathrm{NDG}}$, where $2^R$ denotes the collection of the subsets of $R$. We aim to prove that $|\mathrm{DRR}(R)|/|\mathrm{CD}|\to 1$ as $|R|\to\infty$, or equivalently $|\mathrm{DRR}(R)|/|\mathrm{NDG}(R)|\to\infty$ as $|R|\to\infty$.

Let $S_1$ and $S_2$ be in $2_{\mathrm{DRR}}^R$ and let  $\Gamma_1:=\Cay(R,S_1)$ and $\Gamma_2:=\Cay(R,S_2)$.
Suppose that $\Gamma_1\cong\Gamma_2$ and let $\varphi$ be a digraph isomorphism from $\Gamma_1$ to $\Gamma_2$. Without loss of generality, we may assume that $1^\varphi=1$. Note that $\varphi$ induces a group automorphism from $\Aut(\Gamma_1)=R$ to $\Aut(\Gamma_2)=R$. In particular, $\varphi\in \Aut(R)$ and $S_1$ and $S_2$ are conjugate via an element of $\Aut(R)$. This shows that $$|\mathrm{DRR}(R)|\geq \frac{|2^R_{\mathrm{DRR}}|}{|\Aut(R)|}.$$  Since $|\Aut(R)|\leq  2^{(\log_2(r))^2}$, it follows that
$$|\mathrm{DRR}(R)|\geq \frac{|2^{R}_{\mathrm{DRR}}|}{2^{(\log_2(r))^2}}.$$

Clearly, $|\mathrm{NDG}(R)|\le |2^{R}_{\mathrm{NDG}}|.$
By Theorem~\ref{th:main2}, we have
$$\frac{|\mathrm{DRR}(R)|}{|\mathrm{NDG}(R)|}\ge \frac{(|2^R_{\mathrm{DRR}}|/2^{(\log_2(r))^2})}{|2^R_{\mathrm{NDG}}|}\to \infty,$$
as $|R|\to\infty$. This completes the proof.
\end{proof}

\section{Remarks and comments}\label{sec:comments}
\subsection{Classification of finite simple groups}The work  in this paper is very much in line with the philosophy expressed in the pioneer paper \cite{CAmeron} of Peter Cameron: many interesting problems on finite permutation groups can be  reduced to problems on finite  simple groups, and thus can often be completely solved. For a more recent survey, by Robert Guralnick, one of the leading experts in the applications of the CFSGs, see~\cite{Guralnick}. Clearly, with the CFSG the depth of our understanding of finite simple groups is a function of time, and hence with time deeper and deeper results can be obtained on finite permutation groups and on the symmetries of finite combinatorial structures, provided that one can obtain some sort of reduction to the realm of finite simple groups. When the classification of the finite simple groups was finally announced in 1979 at the Santa Cruz symposium on finite simple groups, many interesting problems on permutation groups were (broadly speaking) immediately trivialized: examples include the classification of the finite $2$-transitive groups \cite{CAmeron} or Sims' conjecture \cite{CPSS}. 

 For this reason, a major theme in current research on finite permutation groups and on group actions on combinatorial structures involves reducing challenging problems in finite permutation groups to questions regarding simple groups.  The heart of our approach for enumerating DRRs and Cayley digraphs are our reductions to questions concerning primitive groups and hence, using the O'Nan-Scott theorem, to questions on simple groups.
  
    There are a number of very interesting questions still widely open where such a reduction might be the key for answering long-standing conjectures. A few of these that are particularly dear to our own hearts are: the enumeration of vertex-transitive graphs, the Polycirculant Conjecture on vertex-transitive graphs \cite{Marusic}, or the Isbell Conjecture on homogeneous games \cite{pablo,isell}.

To avoid misunderstandings, we stress that we are far from saying that \textit{all} interesting problems in finite permutation groups require a reduction to questions about simple groups in order to find a solution or an answer, or that such a reduction is always the most productive or advisable way to work on these problems. Recent work on finite semigroups and synchronizing primitive groups is an example in our opinion where exciting new mathematics is obtained without the CFSG, see \cite{synchronization}.

\subsection{Asymptotic enumeration of vertex primitive Cayley digraphs}
Our proof  of Theorem \ref{th:main1} heavily depends upon the Classification of Finite Simple Groups. However, using the exciting new results of Sun and Wilmes \cite{SunWilmes,wilmes}, generalizing some influential results of paramount importance of  Babai  \cite{babai1, babai2} and Pyber \cite{pyber1}, one can prove the following theorem without invoking the CFSG.
\begin{theorem}\label{th:main23}Let $R$ be a group of order $r$. The proportion of subsets $S$ of $R$ such that $\Aut(\Cay(R,S))$  acts primitively and not regularly on the vertices of $\Cay(R,S)$ tends to $0$ as $r$ tends to $\infty$.
\end{theorem}

\noindent In other words, without the CFSG one might prove (if so minded) that, when $R$ is not a cyclic group of prime order, the automorphism group of a Cayley graph $\Cay(R,S)$ over $R$ admits a non-trivial system of imprimitivity with probility approaching $1$ as $|R|$ tends to $\infty$.

\begin{proof}[Proof of Theorem $\ref{th:main23}$]Observe that the results in Sections~\ref{BG} and~\ref{sec2} do not depend upon the CFSG. Therefore, using Section \ref{sec44} and Definition \ref{def1is}, we are left to prove that
$$\lim_{|R|\to\infty}\frac{|\{S\subseteq R\textrm{ satisfying \ref{hyp1}--\ref{hyp4} in Section \ref{sec44}}\mid \Aut(\Cay(R,S)) \textrm{ primitive}\}|}{2^{|R|}}=0.$$
Let $S\subseteq R$, with $S$ satisfying \ref{hyp1}--\ref{hyp4} in Section \ref{sec44} and with $\Aut(\Cay(R,S))$ primitive. A classical result of  Babai  \cite{babai1, babai2} shows that, if $G$ is a primitive not $2$-transitive group of degree $n$, then $|G|\le 2^{4\sqrt{n}(\log_2 n)^2}$.  Pyber \cite{pyber1} has shown that, if $G$ is a $2$-transitive group of degree $n$ with $\Alt(n)\nleq G$, then $|G|\le 2^{72(\log_2 n)^3}$. Although Pyber's result is not relevant to our situation since a $2$-transitive group of automorphisms for a digraph arises only when the full automorphism group is $\Sym(n)$, the work was highly influential and stimulated further investigation. These results have  been generalized by Sun and Wilmes \cite{SunWilmes,wilmes} motivated by some work in the context of coherent configurations and with a CFSG-free proof. In~\cite[Corollary $1.6$]{SunWilmes}, it is proven that, if $G$ is a primitive permutation group of degree $n$, then either
\begin{enumerate}
\item $|G|\le \exp(O(n^{1/3}\log^{7/3}n))$, or
\item $G$ is $\Sym(n)$ or $\Alt(n)$, or
\item $G$ is $\Sym(m)$ or $\Alt(m)$ where $n={m\choose 2}$ and $G$ is endowed of its primitive action on the $2$-subsets of $\{1,\ldots,m\}$, or
\item $G$ is a subgroup of $\Sym(m)\mathrm{wr}\Sym(2)$ containing $\Alt(m)^2$ where $n=m^2$ and $\Sym(m)\mathrm{wr}\Sym(2)$ is endowed of its natural primitive product action.
\end{enumerate}
The first case does not arise in our context because $|G_1|\ge \exp(O(r^{0.499}))$. In the remaining cases $G$ has rank at most $3$ and hence the proof follows from Lemma~\ref{lemma1}. Since each of the three cases (2)--(4) above contributes at most $8$ groups, and Lemma~\ref{lemma1} tells us that each group comes from at most $8$ connection sets, the numerator (counting the connection sets that aren't accounted for in Sections~\ref{BG} and~\ref{sec2}) is actually bounded by a constant. In fact, a careful examination of the groups and connection sets in these cases reveals that there are at most $8$ connection sets that arise, since some of these connection sets arise in multiple cases, and even multiple times within a case.
\end{proof}

Following the estimates in Sections~\ref{BG} and~\ref{sec2} one can give a quantitative version of Theorem \ref{th:main23}. To obtain a slightly better estimate one has to refine Lemma~\ref{lemma1} in the context of primitive groups. To do so, (using the notation in Lemma~\ref{lemma1}) observe that, if $G\leq \Sym(\Omega)$ is primitive and not regular on $\Omega$, then $\omega$ is the only element of $\Omega$ fixed by each permutation in $G_\omega$  and hence $G$ acts on at most $2^{\frac{|\Omega|+1}{2}}$ digraphs with vertex set $\Omega$.

\subsection{Undirected Cayley graphs}
Our proof of Theorem \ref{th:main1} does not extend to undirected Cayley graphs. Recall that $\Cay(R,S)$ is undirected if and only if $S$ is inverse-closed, that is, $S^{-1}:=\{s^{-1}\mid s\in S\}=S$. While the number of Cayley digraphs on $R$ is $2^{|R|}$, which is a number that depends on the cardinality of $R$ only, the number of undirected Cayley graphs on $R$ is $2^{\frac{|R|+|I(R)|}{2}}$, where $I(R):=\{\iota\in R\mid \iota^2=1\}$, and hence depends on the algebraic structure of $R$.
 
 It turns out that there are only two infinite families of groups that do no admit GRRs. The first family consists of abelian groups of exponent greater than two. If $A$ is such a group and $\iota$ is the automorphism of $A$ mapping every element to its inverse, then every Cayley graph on $A$ admits $A\rtimes\langle \iota\rangle$ as a group of automorphisms. Since $A$ has exponent greater than $2$, $\iota\ne 1$ and hence no Cayley graph on $A$ is a GRR. The other groups that do not admit GRRs are the generalised dicyclic groups, see \cite[Definition $1.1$]{MSV} for a definition.
 
 It was proved by Godsil that abelian groups of exponent greater than $2$ and generalised dicyclic groups are the only two infinite families of groups that do not admit GRRs. The stronger Conjecture~\ref{conjecture...} was made (at various times) by Babai, Godsil, Imrich and Lov\'{a}sz.
 
 \begin{conjecture}[see \cite{BaGo}, Conjecture $2.1$ and \cite{Go2}, Conjecture $3.13$]\label{conjecture...}Let $R$ be a group of order $r$ which is neither generalised dicyclic nor abelian. The proportion of inverse-closed subsets $S$ of $R$ such that $\Cay(R,S)$ is a $\mathrm{GRR}$ goes to $1$ as $r\to \infty$.
 \end{conjecture}
 There are two places where our proofs do not immediately (or with some work) extend to undirected graphs, namely Lemmas \ref{lemma4242} and \ref{lemma1}. In these two lemmas, which are pivotal for our reductions, the fact that we are dealing with arbitrary Cayley digraphs seems to be essential. To be more precise, the proof of each of these lemmas generalises perfectly to the undirected case, but the resulting bounds do not produce a negligible fraction of all undirected Cayley graphs except for groups where $I(R)$ is very large; that is, groups that have many involutions. Currently we have no idea in how to fix this problem, that on the surface seems to be purely technical: the undirected case is much much harder on a technical level,  but conceptually not very different. Here we simply mention two papers \cite{DFR,GFR}, which inspired the work in this paper. In turn, \cite{DFR,GFR} owe a lot to the work of Imrich, Nowitz and Watkins in \cite{IW,NW1,NW2,NW3}. The first paper~\cite{DFR} deals with the enumeration of digraphical Frobenius representations and the second~\cite{GFR} deals with graphical (and hence undirected) Frobenius representations. Thus \cite{DFR} could be compared with the work in this paper and  \cite{GFR} could be compared to the asymptotic enumeration of undirected Cayley graphs (though the analogy does not run very deep). The key 
strategy in \cite{GFR} for generalizing~\cite{DFR} to undirected Cayley graphs is to use a dichotomy argument: subdivide arbitrary groups $R$ in two classes, the first class formed by the groups that do not admit any automorphisms inverting many elements and the second class formed by the groups that do admit such an automorphism. (We are deliberately vague about the precise meaning of ``many'' here, because in our new context we have no clear idea of what ``many'' might mean.)  The groups falling into the first class are dealt with ``probabilistic'' arguments, whereas the groups in the second class have a highly restricted structure and hence can be analysed with ad-hoc arguments. We hope that in the future a similar approach could also be used for asymptotically enumerating undirected Cayley graphs and hence resolving Conjecture~\ref{conjecture...}.

 \subsection{Vertex-transitive digraphs}Some of the arguments in this paper generalize, again with no work or with only little work, to the problem of asymptotic enumeration of vertex-transitive digraphs on $n$ vertices (up to isomorphism). Using the same approach as in this paper and in particular Lemma \ref{lemma1}, in order to enumerate vertex-transitive digraphs  it seems natural and important to asymptotically estimate (up to conjugation in $\Sym(n)$)  one of the following classes of transitive permutation groups:
 \begin{itemize}
 \item \textit{minimally transitive groups}, that is, transitive subgroups $G$ of $\Sym(n)$ with the property that each proper subgroup of $G$ is intransitive;
 \item \textit{transitive $2$-closed groups}.
 \end{itemize}
Indeed, suppose as a running conjecture that one of the previous two classes of permutation groups has at most $2^{o(n)}$ elements. Just to make these ideas clear, let us assume that the number of $2$-closed subgroups of $\Sym(n)$ up to conjugation is at most $2^{o(n)}$. This seems a reasonable conjecture to make: the regular subgroups of $\Sym(n)$ are $2$-closed and, up to conjugation, they are in one-to-one correspondence with the groups of order $n$ up to isomorphism. Pyber \cite{pyber} has shown that there are at most $n^{\left(\frac{2}{27}+o(1)\right)\mu(n)^2}$ groups of order $n$, where $\mu(n)=\max_{i=1}^kg_i$, $n=\prod_{i=1}^kp_i^{g_i}$ and $p_1,\ldots,p_k$ are distinct primes. Therefore, we have only at most $n^{\log_2(n)^2}\le 2^{(\log_2 n)^3}$ regular subgroups up to conjugation. Our wishful thinking requires that there are also at most $2^{o(n)}$ transitive subgroups of $\Sym(n)$ which are $2$-closed and not regular. If this happens to be true,   applying Lemma \ref{lemma1} allows us to conclude that there are at most $2^{3n/4+o(n)}$ vertex-transitive digraphs on $n$ vertices that are not Cayley digraphs. Since Theorem~\ref{th:main1} shows that we have at least $2^{n+o(n)}$ Cayley digraphs on $n$ vertices, we deduce that most vertex-transitive graphs are Cayley digraphs (actually DRRs), thus answering a question of McKay and Praeger~\cite[page 54]{MP}. A little bit of evidence that this approach has potential is given by Pyber~\cite[Theorem~4.4]{pyber2}.
  
\subsection{Normal Cayley graphs}
A Cayley (di)graph $\Gamma$ of $G$ is said to be a {\it normal} Cayley (di)graph of $R$ if the regular representation of $R$ is normal in $\Aut(\Gamma)$.  Xu conjectured that almost all Cayley (di)graphs of $R$ are normal Cayley (di)graphs of $R$; we have given his formulation more precisely in Theorem~\ref{th:normal}.  As noted, we have proven the directed version of this conjecture as any DRR on $R$  has automorphism group $R$ and hence it is a normal Cayley digraph of $R$. Similar results for undirected graphs, supporting the conjectures of Xu are proved in \cite{DSV,MSV}, when $R$ is an abelian group and when $R$ is a dicyclic group.

We find that, in principle and very likely in practise, Conjecture \ref{conjecture...} and the undirected conjecture of Xu are very similar. Indeed, requiring that $\Cay(R,S)$ is a normal Cayley graph means requiring that  $\Aut(\Cay(R,S))\le R\rtimes \Aut(R)$. Now, since $|\Aut(R)|\le 2^{(\log_2(|R|))^2}$ is small compared to the number of Cayley graphs, it is reasonable to expect that most Cayley graphs on  $R$ are GRRs if and only if most Cayley graphs on $R$ are normal. The forward implication is clear, because each GRR is a normal Cayley graph.

\subsection{Asymtotic enumeration of vertex-transitive graphs and Cayley graphs of bounded valency}
There is another problem we would like to mention. Let $d$ be a positive number. The  asymptotic enumeration of vertex-transitive graphs and of Cayley graphs of valency $d$ is a widely open question that has  hardly been touched so far. In this context, in our opinion it is more interesting and natural to consider only connected graphs; this also avoids degeneracies. The case $d=2$ is trivial. Thus the first interesting case is $d=3$ and this already offers intricate questions in group generation. The best result  for  $d=3$ is Theorem $1.2$ in \cite{PSV} where (roughly speaking) it is proved that the functions counting  the number of GRRs, Cayley graphs, and vertex-transitive graphs of valency $3$ and up to $n$ vertices are asymptotically very similar. Surprisingly, the same result holds if ``vertex-transitive'' is replaced with the much stronger requirement of the graphs being ``$5$-arc-transitive''. 

To prove analogous results for arbitrary valencies following the arguments in \cite{PSV}, it seems important to have a strong understanding of certain transitive subgroups of $\Sym(n)$. In this context, we pose a conjecture. First, however, we need to establish the setting. Let $G$ be a transitive subgroup of $\Sym(n)$ and let $\omega\in\{1,\ldots,n\}$. Let $O_1,\ldots,O_\kappa$ be the orbits of $G_\omega$ on $\{1,\ldots,n\}$. For each $i\in \{1,\ldots,\kappa\}$, there is a digraph $\Gamma_i$ associated to $O_i$ called the \textit{orbital} digraph for $G$: the vertex set of $\Gamma_i$ is $\{1,\ldots,n\}$ and the arc set of $\Gamma_i$ is $\{(\omega,\delta)^g\mid g\in G,\delta\in O_i\}$. For each subset $I\subseteq\{1,\ldots,\kappa\}$, we may associate a \textit{merged orbital} digraph $\Gamma_I$ where the vertex set is again $\{1,\ldots,n\}$ and the arc set is $\{(\omega,\delta)^g\mid g\in G,\delta\in O_i,i\in I\}$. It is clear that the merged orbital digraphs of $G$ are exactly the digraphs $\Gamma$ with vertex set $\{1,\ldots,n\}$ and with $G\le \Aut(\Gamma)$. 

\begin{conjecture}
{\rm There exists a function $f:\mathbb{N}\to\mathbb{N}$ such that the number of transitive groups  of degree $n$ (up to conjugation in $\Sym(n)$) admitting a connected merged digraph of valency $d$ is at most $n^{f(d)\log n}$}.
\end{conjecture}
This conjecture is trivially true using Sims' conjecture, if the group $G$ is primitive. It is also true when $d\le 3$ by the work in \cite{PSV}. Much is known about the generation of the transitive subgroups of $\Sym(n)$, see for instance \cite{lucchini1,lucchini2,federico}. However, there seems to have been no investigation into the number of generators that are necessary for a transitive subgroup $G$ of $\Sym(n)$ where some information on the merged orbitals of $G$ is given.

\subsection{Bipartite regular representations}
Now that we have established that most Cayley digraphs are DRRs, there are other natural questions that arise. For instance, suppose that $R$ has subgroups having index $2$, is it true that most bipartite Cayley digraphs on $R$ are DRRs? A partial answer to this question (only in the case of abelian groups)  is given in~\cite{Jiali}.

\end{document}